\documentclass[aap,MSNbibl,dvips]{arximspdf}
\usepackage{graphicx}

\doi{10.1214/13-AAP999} 
\volume{25}
\issue{1}
\pubyear{2015}
\firstpage{373}
\lastpage{405}

\makeatletter
\renewcommand{\mid}{|}

\newcommand{\rright}{\right}
\newcommand{\lleft}{\left}
\newcommand{\rrVert}{\Vert}
\newcommand{\llVert}{\Vert}
\newtheorem{teo}{Theorem}[section]
\newtheorem{prop}[teo]{Proposition}
\newtheorem{lem}[teo]{Lemma}
\newtheorem{clam}[teo]{Claim}
\newproclaim{rem}[teo]{Remark}
\newproclaim{df}[teo]{Definition}
\makeatother

\begin{document}
\begin{frontmatter}

\title{Limiting geodesics for first-passage percolation on~subsets of
$\mathbb{Z}^2$}
\runtitle{Limiting geodesics for first-passage percolation}

\begin{aug}
\author[A]{\fnms{Antonio} \snm{Auffinger}\ead[label=e1]{auffing@uchicago.edu}},
\author[B]{\fnms{Michael} \snm{Damron}\corref{}\ead[label=e2]{mdamron@indiana.edu}\thanksref{T2}}
\and
\author[B]{\fnms{Jack} \snm{Hanson}\ead[label=e3]{jthanson@indiana.edu}\thanksref{T3}}
\runauthor{A. Auffinger, M. Damron and J. Hanson}
\affiliation{University of Chicago, Princeton University and Princeton University}
\address[A]{A. Auffinger\\
Department of Mathematics\\
University of Chicago\\
5734 S. University Avenue\\
Chicago, Illinois 60637\\
USA\\
\printead{e1}}
\address[B]{M. Damron\\
J. Hanson\\
Department of Mathematics\\
Indiana University\\
831 E. 3rd St.\\
Bloomington, Indiana 47405\\
USA\\
\printead{e2}\\
\phantom{E-mail:\ }\printead*{e3}}
\end{aug}
\thankstext{T2}{Supported by NSF Grants DMS-09-01534 and DMS-10-07626.}
\thankstext{T3}{Supported by an NSF graduate fellowship and NSF Grant
PHY-1104596.}

\received{\smonth{2} \syear{2013}}
\revised{\smonth{8} \syear{2013}}

%
\begin{abstract}
It is an open problem to show that in two-dimensional first-passage
percolation, the sequence of finite geodesics from any point to $(n,0)$
has a limit in $n$. In this paper, we consider this question for
first-passage percolation on a wide class of subgraphs of $\mathbb
{Z}^2$: those whose vertex set is infinite and connected with an
infinite connected complement. This includes, for instance, slit
planes, half-planes and sectors. Writing $x_n$ for the sequence of
boundary vertices, we show that the sequence of geodesics from any
point to $x_n$ has an almost sure limit assuming only existence of
finite geodesics. For all passage-time configurations, we show
existence of a limiting Busemann function. Specializing to the case of
the half-plane, we prove that the limiting geodesic graph has one
topological end; that is, all its infinite geodesics coalesce, and
there are no backward infinite paths. To do this, we prove in the
Appendix existence of geodesics for all product measures in our domains
and remove the moment assumption of the Wehr--Woo theorem on absence of
bigeodesics in the half-plane.
\end{abstract}

%

\end{frontmatter}

\section{Introduction}\label{sec1}

First-passage percolation may be regarded as a family of models, each
of which yields a random pseudo-metric on a graph. It was introduced by
Hammersley and Welsh \cite{HW65} as a model for the passage of a fluid
through a porous medium and it has provided many interesting problems
to the probability and statistical physics community. It also has links
to classical physics through disordered Ising models \cite{FLN,HH}
and to mathematical biology through the study of spread of infections
and competition models \cite{Richardson}.

The main goal is to understand the (properly scaled) random geometry
induced by the pseudo-metric. This has been achieved in two (not
necessarily unrelated) ways: first, by studying the asymptotics and
fluctuations of the distance function between two points of diverging
graph distance; second, by understanding the structure of finite or
infinite geodesics, length minimizing paths in this pseudo-metric. This
paper addresses questions in the latter group.

The study of geodesics in first-passage percolation starts with Newman
\cite{N95}, Licea--Newman \cite{LN96} and Wehr \cite{Wehr}. It was
conjectured that every semi-infinite geodesic should have an asymptotic
direction and all such geodesics with a given fixed direction should
merge. These statements were established in \cite{N95,LN96} under
certain strong assumptions on the limit shape, the $t\to\infty$
scaling limit of the random ball of size $t$ of the origin. Although
natural and expected, these assumptions have not been verified.

The analysis of geodesics continues with the work of H\"aggstr\"om--Pemantle \cite{HP}, Garet--Marchand \cite{GM}, Hoffman \cite
{Hoffman1,H08}, Damron--Hochman \cite{DH} and Auf\-finger--Damron \cite
{AD11}. They establish existence of a wide class of first-passage
percolation processes with infinitely many disjoint infinite one-sided
geo\-desics. All these results explored-known properties of the limit
shape or a particular choice of passage-time distribution. Under
minimal assumptions, however, Damron--Hanson \cite{DH12} recently
proved some forms of Newman's conjectures. They establish almost-sure
coalescence of distributional limits of geodesics and nonexistence of
certain infinite backward paths. Despite these advances, it is still an
open problem to show that in two dimensions, the sequence of finite
geodesics from any point to the points $(n,0)$ has a limit.

In this manuscript, we consider this question on infinite subgraphs of
$\mathbb Z^2$. Assuming only existence of finite geodesics, we show
that sequences of finite geodesics from any point to boundary points
have almost sure limits. Our method is motivated by the ``paths
crossing'' trick of Alm and Wierman \cite{AW99}. In the case of the
half-plane, we prove the limiting geodesic graph has one topological
end; that is, all its infinite geodesics coalesce and there are no
backward infinite paths. To our knowledge, this is the first time that
limiting geodesics are shown to exist under minimal assumptions on the
passage times.

We close this section by commenting on limitations of our arguments and
speculations for further advances. The crucial use of the boundary is
to allow the paths crossing argument of Claim~\ref{clamgeodesicintersection}. In the full plane,
this is not possible. Even if one leaves the boundary, taking, for
example, $(0,n)$ in the upper half-plane, this argument breaks down.
Furthermore, the analysis of half-plane geodesics in this paper heavily
uses horizontal translation invariance of the passage-time
distribution. This is required to apply the ergodic theorem at several
points throughout the arguments. So in many other domains (e.g., quarter planes or sectors) we do not know if the geodesics
constructed here coalesce, although it is reasonable to expect them to.

\subsection{Outline of the paper}\label{sec1.1}
In the rest of the \hyperref[sec1]{Introduction}, we give the precise definition of the
model, and we state the main theorems of the paper. In Section~\ref{sec2} we
establish, without any assumption on the passage times, existence of
limits for Busemann functions. Under the hypothesis of existence of
finite geodesics, in Section~\ref{sec3}, we prove existence of the limiting
geodesic graph. In Section~\ref{sec4}, we show that in the upper half-plane,
this limiting graph has one end, establishing coalescence of any pair
of its infinite geodesics. We finish the paper with three \hyperref[secdualedge]{Appendices}.
In Appendix~\ref{secdualedge}, we give an alternate characterization of our domains.
Appendix \ref{secexistence} proves the existence of finite geodesics for all product
measures. Appendix \ref{secWW} extends the Wehr--Woo theorem \cite{WW98} on
absence of doubly infinite geodesics in the half-plane to more general measures.

\subsection{Definitions}\label{sec1.2}
Let $(\mathbb{Z}^2, \mathcal{E}^2)$ denote the square lattice with
nearest-neighbor edges. We consider first-passage percolation on
particular infinite subsets of this graph. Let $V \subseteq\mathbb
{Z}^2$ be a connected [in $(\mathbb{Z}^2,\mathcal{E}^2)$] infinite
set whose complement is also connected and infinite. Write $E$ for the
set of edges with both endpoints in~$V$. We will need the graph dual to
the square lattice, the vertex set of which is $ (\mathbb
{Z}^2 )^* = \mathbb{Z}^2 + (1/2,1/2)$ and the edge set of which
is $ (\mathcal{E}^2 )^* = \mathcal{E}^2 + (1/2,1/2)$. The
edge $e^*$ is said to be dual to $e \in\mathcal{E}^2$ if it bisects $e$.
We prove in Appendix~\ref{secdualedge} that there exists some path
of dual edges
%
\begin{equation}
\label{defGamma} \Upsilon= \bigl(e_i^*\bigr)_{i \in\mathbb{Z}}
\end{equation}
which does not (vertex) self-intersect and such that $(V,E)$ is one of
the two components of the graph formed from $(\mathbb{Z}^2, \mathcal
{E}^2)$ by removing the edges $(e_i)$ dual to those edges $(e_i^*)$.
%
\begin{equation}
\label{defvi} \mbox{Let }v_i \mbox{ be the endpoint of
}e_i\mbox{ that lies in }V.
\end{equation}
Note that while $\Upsilon$ is not self-intersecting, a particular
$v_i$ may appear multiple times (at most 3 times).

We do first-passage percolation in $(V,E)$ by setting $\Omega=
[0,\infty)^E$ and denoting a typical element of $\Omega$ by $\omega=
(\omega_e)_{e \in E}$. For $x,y \in V$, a path from $x$ to $y$ in $V$
is a sequence of alternating vertices and edges
\[
x=w_0,e_0,w_1, \ldots,w_{n-1},e_{n-1},w_n=y
\]
such that for all $i$, $e_i = \{w_i, w_{i+1}\} \in E$. Clearly a path
is uniquely determined by its sequence of vertices or its sequence of
edges, so we will at times refer to it in one of these ways. We will
write $\gamma\dvtx x \leadsto y$ to denote that $\gamma$ is a path from
$x$ to $y$. We will use $\|\cdot\|_1$ to denote the $l^1$ norm.

The resulting passage time is written $\tau$. That is, $\tau(\gamma)
= \sum_{e \in\gamma} \omega_e$ is the passage time of a finite path
$\gamma$ in $(V,E)$ and $\tau(x,y) = \inf_{\gamma\dvtx  x\leadsto y}
\tau(\gamma)$ is the passage time between $x$ and $y$ in $V$. As
defined, $\tau$ is a pseudo-metric. A geodesic from $x$ to $y$ is a
path $\gamma\dvtx  x \leadsto y$ in $(V,E)$ such that $\tau(\gamma) =
\tau(x,y)$. Note that if there exists a geodesic between some pair of
points, there is at least one vertex self-avoiding geodesic.

We will define (for $x$ and $y$ elements of $V$) the Busemann function
\[
B_n(x,y) = \tau(x,v_n) - \tau(y,v_n).
\]

\subsection{Main results}\label{sec1.3}

\subsubsection{Arbitrary $(V,E)$}\label{sec1.3.1}
The first result shows that asymptotic limits of the $(B_n)$ exist
under no assumptions on $\omega$. That is, it holds for all passage-time configurations.

\begin{teo}\label{teobusemannlimits}
For any $x,y \in V$ and $\omega\in\Omega$,
%
\begin{equation}
\label{eqbusemannlimits} B(x,y):= \lim_{n \to\infty} B_n(x,y)\qquad
\mbox{exists}.
\end{equation}
\end{teo}

\begin{rem}
We strongly believe that Busemann limits exist in wide generality (in
particular, even in the full-plane), but we do not have a proof. That
is, we expect that for any $\theta\in[0,2\pi)$ and any sequence
$(x_n)$ of vertices in $\mathbb{Z}^2$ such that $\arg x_n \to\theta$
with $x_n \to\infty$, the limit $\tau(x,x_n) - \tau(y,x_n)$ exists
almost surely for $x,y \in\mathbb{Z}^2$.
\end{rem}

For the second result we consider a measure $\mathbb{P}$ on $\Omega$
(with the product Borel sigma algebra) that admits geodesics; that is,
\[
\mathbb{P}(\exists\mbox{ a geodesic } \gamma\dvtx x \leadsto y) = 1\qquad\mbox{for
all } x,y \in V.
\]
Under this condition we can associate to almost every $\omega\in
\Omega$ and each $n \in\mathbb{Z}$ a geodesic graph $\mathbb{G}_n =
\mathbb{G}_n(\omega)$. This is a directed graph with vertex set $V$
built from a configuration $\eta_n=\eta_n(\omega)$ from the space $\{
0,1\}^{\vec{E}}$, where $\vec{E}$ is the set of directed edges
corresponding to $E$,
\[
\vec{E} = \bigl\{(x,y)\dvtx  \{x,y\} \in E\bigr\}.
\]
The definition of $\eta_n$ is as follows. We set $\eta_n((x,y)) = 1$
if $\{x,y\}$ is in a geodesic from some vertex in $V$ to $v_n$ and
$\tau(x,v_n) \geq\tau(y,v_n)$. Otherwise we set $\eta_n((x,y))=0$.
The graph $\mathbb{G}_n$ is then induced by its directed edge set, the
set of $e$ such that $\eta_n(e) = 1$.

We say that $\eta_n \to\eta\in\{0,1\}^{\vec{E}}$ if for each $e
\in\vec{E}$, $\eta_n(e) \to\eta(e)$. In this case we write
$\mathbb{G}_n \to\mathbb{G}$, where $\mathbb{G}$ is the directed
graph corresponding to $\eta$.

\begin{teo}\label{teographconvergence}
Suppose that $\mathbb{P}$ admits geodesics. Then with probability one,
$(\mathbb{G}_n)$ converges to a graph $\mathbb{G}$. Each directed
path in $\mathbb{G}$ is a geodesic.
\end{teo}

\subsubsection{On the half-plane $\mathbb{H}$}\label{sec1.3.2}

Taking the vertex set $V = V_H = \{(x_1,x_2) \in\mathbb{Z}^2\dvtx  x_2
\geq0\}$ and $E_H$ the induced set of edges, we can analyze
first-passage percolation more closely on $\mathbb{H} = (V_H,E_H)$,
taking advantage of translation invariance of standard measures. The
relevant space is $\Omega_H = [0,\infty)^{E_H}$ and we define a
family of translation operators $\{T_x\dvtx  x \in V_H\}$ on $\Omega_H$ by
\[
(T_x \omega)_e = \omega_{e+x},
\]
where if $e=\{v,w\}$ then $e+x = \{v+x,w+x\}$.

For the results in this section we will consider a probability measure
$\mathbb{P}$ satisfying one of two assumptions, labeled \textup{(A)}~and~\textup{(B)} below. Assumption~\textup{(B)} includes the upward finite energy
property from \cite{DH12}:

\begin{df}
Given an edge set $E'$, a Borel probability measure $\mathbb{P}$ on
$[0,\infty)^{E'}$ satisfies the upward finite energy property if for
each $e \in E'$ and $\lambda$ such that $\mathbb{P}(\omega_e \geq
\lambda)>0$, we have
\[
\mathbb{P}(\omega_e \geq\lambda\mid\check{\omega}) > 0\qquad
\mbox{almost surely}.
\]
\end{df}

In the definition we have used the notation $\omega= (\omega_e,
\check{\omega})$, where $\check{\omega} = (\omega_f\dvtx\break  f \neq e)$.

The assumptions we need are:
\begin{longlist}[(A)]
\item[(A)] $\mathbb{P}$ is a product measure with continuous
marginals, or
\item[(B)] $\mathbb{P}$ is the restriction to $[0,\infty
)^{E_H}$ of a Borel probability measure $\widehat{\mathbb{P}}$ on
$[0,\infty)^{\mathcal{E}^2}$ that satisfies the upward finite energy
property and the assumptions of Hoffman \cite{H08}:
\begin{enumerate}[(a)]
\item[(a)]$\widehat{\mathbb{P}}$ is ergodic relative to the translations
$T_x$ for $x \in\mathbb{Z}^2$;
\item[(b)]$\widehat{\mathbb{P}}$ has all the symmetries of $\mathbb{Z}^2$;
\item[(c)]$\int\omega_e^{2+\alpha} \,\mathrm{d}\widehat{\mathbb{P}}<\infty
$ for some $\alpha>0$;\vspace*{2pt}
\item[(d)]$\widehat{\mathbb{P}}$ has unique passage times: with probability
one, no two (edge) nonempty distinct paths have the same passage time and\vspace*{1pt}
\item[(e)] the limiting shape for $\widehat{\mathbb{P}}$ is bounded.
\end{enumerate}
\end{longlist}
Under\vspace*{1pt} parts (a)--(c) of assumption \textup{(B)}, Kingman's theorem implies
that if we write $\tau'$ for the passage time in $\mathbb{Z}^2$, then
for each $y \in\mathbb{Z}^2$, the limit $g(y) = \lim_{n \to\infty}
\tau'(0,ny)/n$ exists almost surely and in $L^1$. Part~(b) is required
for the geodesic graph to be a forest. This is used several times in
the final arguments. So our arguments do not apply, for instance, to
geometric\vspace*{1pt} weights. Part~(e) of assumption \textup{(B)} is then the statement
that $\inf_{y \neq0} \frac{g(y)}{\|y\|_1}>0$.

Under either of these assumptions, one can show that $\mathbb{P}$
admits geodesics. Under~\textup{(A)}, we show it in Appendix~\ref{secexistence} for general graphs $(V,E)$ considered in this paper. Under
\textup{(B)} it follows from the shape theorem proved by Boivin \cite{B90}
and boundedness of the limit shape. This means we can use the results
from the previous subsection. For the statement of the main theorem, we
use the shorthand $x \to y$ for vertices $x,y$ in a directed graph
$\vec{G}$ if there is a directed path from $x$ to~$y$ in $\vec{G}$.

\begin{teo}\label{teohalfplanegeodesics}
Assume \textup{(A)} or \textup{(B)}. Writing $x_n = (n,0)$, the geodesic graphs
$(\mathbb{G}_n)$ converge almost surely to a directed graph $\mathbb
{G}$ with the following properties:
\begin{longlist}[(3)]
\item[(1)] each vertex in $V_H$ has out-degree 1;
\item[(2)] viewed as an undirected graph, $\mathbb{G}$ has no circuits;
\item[(3)]for each $x \in V_H$, the backward cluster $B_x = \{y \in
V_H\dvtx  y \to x\}$ is finite;
\item[(4)] writing $\Gamma_x$ for the unique self-avoiding infinite
directed path in $\mathbb{G}$ starting from $x$, for all $x,y \in
V_H$, $\Gamma_x$ and $\Gamma_y$ coalesce. That is, their edge
symmetric difference is finite.
\end{longlist}
\end{teo}

\begin{rem}
It is an important problem to show that the geodesics constructed above
have direction $\mathbf{e}_1$. We believe this is true; however, we
cannot prove~it.
\end{rem}

\section{Existence of Busemann limits}\label{sec2}

The main goal of this section is prove Theorem~\ref{teobusemannlimits}. We begin with $x,y \in\{v_i\}_{i \in\mathbb{Z}}$,
defined in (\ref{defvi}).

\begin{prop}
\label{propbusemannaxis}
For any $x, y \in\{v_i\}_i$ and $\omega\in\Omega$, the limit in
(\ref{eqbusemannlimits}) exists. Moreover, the convergence is monotone.
\end{prop}

\begin{pf}
We assume that $x = v_i$ and $y = v_j$ for $i < j$, and we let
$\varepsilon>0$. Fix any $n_2>n_1 > j$ such that $v_{n_1} \neq v_{n_2}$.
We can now choose vertex self-avoiding paths $\gamma\dvtx  x \leadsto
v_{n_1}$ and $\gamma'\dvtx  y \leadsto v_{n_2}$ to satisfy
\[
\tau(\gamma) \leq\tau(x,v_{n_1}) + \varepsilon\quad\mbox{and}\quad \tau
\bigl(\gamma'\bigr) \leq\tau(y,v_{n_2}) + \varepsilon.
\]
Form a continuous path $\beta$ (in $\mathbb{R}^2$) by taking $\gamma
$, adjoining half of the edge $e_{n_1}$, adjoining the segment of
$\Upsilon$ [recall the definition from (\ref{defGamma})] between
$e_{n_1}^*$ and $e_i^*$, and then finally appending half of the edge
$e_i$, to form a continuous circuit based at $x$. Since this circuit is
a Jordan curve, it separates $\mathbb{R}^2$ into an interior and an
exterior. See Figure~\ref{figfigurejordan} for an illustration of
$\beta$.

\begin{figure}[t]

\includegraphics{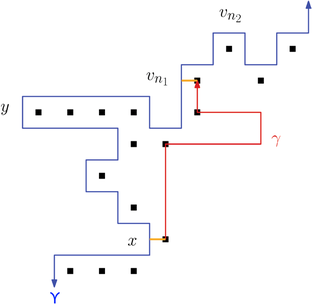}

\caption{Construction of the Jordan curve $\beta$. It consists of the
right path $\gamma$, two half-edges connecting $\gamma$ to the left
path, which is a segment of $\Upsilon$ between $v_{n_1}$ and
$x$.}\label{figfigurejordan}
\end{figure}

Our first observation is that either $y \in\beta$ or $y$ is in the
interior of $\beta$ (and in fact, $y \in\beta$ only if $y \in\gamma
$). The reason is that $y$ is an endpoint of one of the $e_i$'s, which
must cross $\beta$. Since the other endpoint of this edge is in $V^c$,
it cannot be in the interior of $\beta$ (or on $\beta$). The Jordan
curve theorem implies that these endpoints are in different components,
and thus if $y \notin\beta$, it must be in the interior of $\beta$.
We make the following claim:

\begin{clam}
\label{clamgeodesicintersection}
$\gamma' \cap\gamma$ contains a vertex of $\mathbb{Z}^2$.
\end{clam}

To show the claim, we first prove that $v_{n_2}$ is either on $\beta$
or in the exterior of~$\beta$. Accordingly, assume $v_{n_2}$ is not on
$\beta$. Notice that neither endpoint of $e_{n_2}$ can touch~$\beta$.
Furthermore the edge $e_{n_2}$ cannot intersect $\beta$ because
$e_{n_2}^*$ is not contained in $\beta$. Therefore both endpoints are
in the same component of the complement of $\beta$ and since the other
one is in $V^c$, they must be in the exterior of $\beta$.

Now, considering $\gamma'$ as a continuous plane curve, we note that
$\gamma'$ must intersect~$\beta$ (since it has to reach $v_{n_2}$,
which is not in the interior of $\beta$), but it cannot intersect
$\Upsilon$. Therefore, it must intersect $\gamma$; this intersection
must happen at a vertex, though it may of course also happen at one or
more edges. This proves the claim.

We will complete the existence proof for the limit in (\ref{eqbusemannlimits}) by showing that $B_n(x,y)$ is monotone in $n$ for
fixed $x$ and $y$. Let $n_1$ and $n_2$ be as above. For any path
$\sigma: a \leadsto b$ and $c \in\sigma$ write $\sigma\mid_c$ for
the segment of $\sigma$ from the first meeting of $c$ onward and
$\sigma\mid^c$ for the segment of $\sigma$ to the first meeting of
$c$. Then letting $w$ be a point in $\gamma' \cap\gamma$,
\begin{eqnarray*}
\tau(x,v_{n_2}) + \tau(y,v_{n_1}) &\leq& \bigl[ \tau\bigl(
\gamma\mid^w\bigr) + \tau\bigl(\gamma'
\mid_w\bigr) \bigr] + \bigl[ \tau\bigl(\gamma'
\mid^w\bigr) + \tau(\gamma\mid_w) \bigr]
\\
&=& \bigl[ \tau\bigl(\gamma\mid^w\bigr) + \tau(\gamma
\mid_w) \bigr] + \bigl[ \tau\bigl(\gamma'
\mid^w\bigr) + \tau\bigl(\gamma' \mid_w
\bigr) \bigr]
\\
&=& \tau(\gamma) + \tau\bigl(\gamma'\bigr) \leq
\tau(x,v_{n_1}) + \tau (y,v_{n_2}) + 2 \varepsilon.
\end{eqnarray*}
Taking $\varepsilon\to0$,
%
\begin{equation}
\label{eqmonotonetimes} \tau(x,v_{n_2}) + \tau(y,v_{n_1}) \leq
\tau(x,v_{n_1}) + \tau(y,v_{n_2}).
\end{equation}
We can rearrange the terms in (\ref{eqmonotonetimes}) to find that
\[
B_{n_2}(x,y) \leq B_{n_1}(x,y).
\]
Since $B_n(x,y)$ is a sequence bounded below by $-\tau(x,y)$, $\lim
B_n(x,y)$ exists.
\end{pf}

We now move on to general $x,y \in V$ and prove the limit in (\ref{eqbusemannlimits}) exists. We will need a few geometric notions. Let
$\alpha$ denote the vertex set of a finite, connected subgraph of
$(V,E)$ which contains some $v_i$. Denote by $(V',E')$ the graph formed
by setting $V' = V \setminus\alpha$ and letting $E'$ be formed from
$E$ by removing every edge with an endpoint in $\alpha$. The graph
$(V',E')$ may have multiple components, but the following claim allows
us to find a single component defining the Busemann function.

\begin{clam}
\label{clamcomponents}
There exists a component $(\overline{V},\overline{E})$ of $(V',E')$
and an $M < \infty$ such that, for all $n > M$,
$v_n \in\overline{V}$. Moreover, $(\overline{V},\overline{E})$ is
formed\vspace*{1pt} from $(\mathbb{Z}^2,\mathcal{E}^2)$ by the removal of edges
dual to a doubly infinite, self-avoiding path $\overline{\Upsilon}$
in the dual lattice.
\end{clam}

\begin{pf}
Note that if $v_n \neq v_{n+1}$, then there exists a path in $(V,E)$
between $v_n$ and $v_{n+1}$ of Euclidean length at most two. Since $\|
v_n\|_1 \rightarrow\infty$, we can choose $M$ such that
\[
\operatorname{dist}\bigl(\{v_n\}_{n > M}, \alpha\bigr) \geq2,
\]
where $\operatorname{dist}(\cdot,\cdot)$ is the $(V,E)$ graph distance.
Then $\{v_n\}_{n > M}$ must all lie in one component of $(V',E')$,
which we denote by $(\overline{V},\overline{E})$.

It remains to show that $(\overline{V},\overline{E})$ can be formed
from $(\mathbb{Z}^2, \mathcal{E}^2)$ by cutting along a doubly
infinite, loop-free dual path $\overline{\Upsilon}$. By
Proposition~\ref{propdualedge} in Appendix~\ref{secdualedge},
it suffices to show that both $\overline{V}$ and $\mathbb{Z}^2
\setminus\overline{V}$ are infinite and connected (as subsets of~$\mathbb{Z}^2$). Both claims are true for $\overline{V}$. Moreover,
$\mathbb{Z}^2 \setminus\overline{V}$ is infinite, since it contains~$V^c$. Because $\alpha$ is connected and contains a point of $\{v_i\}
_i$, we see that $\mathbb{Z}^2 \setminus\overline{V}$ is connected;
it consists of the union of $\alpha$, $V^c$ and the sites of $V$ which
were only reachable from the large $v_n$'s via sites of $\alpha$; see
Figure~\ref{figalpha}. Therefore, by the above, the dual edge
boundary between $\overline{V}$ and $\mathbb{Z}^2 \setminus\overline
{V}$ is a doubly infinite self-avoiding dual path, proving the claim.
\end{pf}

We note that, by Proposition \ref{propbusemannaxis} and the
linearity of the Busemann function, we need only prove the existence of
the limit in (\ref{eqbusemannlimits}) when $y\notin\{v_i\}_i$ but
$x$ is some $v_m$ (which can be chosen as a function of $y$). Fix $y$,
and denote by $\alpha$ the vertex set of some (vertex self-avoiding,
finite) path in $(V,E)$ which starts at a vertex adjacent to $y$ and
ends at a vertex $v_m\in\{v_i\}_i$. Form the graph $(\overline{V},
\overline{E})$ as in Claim \ref{clamcomponents}; denote by
$\overline{\Upsilon}$ the doubly-infinite dual path whose existence
is established in the claim, and define $\{\bar{v}_i\}_i$ analogously
to $\{v_i\}_i$. We may choose an orientation of $\{\bar{v}_i\}_i$ such
that the following holds. There exists $\kappa\in\mathbb{Z}$ such
that for all large $n$, $v_n=\bar{v}_{n+\kappa}$.\vspace*{2pt}

\begin{figure}

\includegraphics{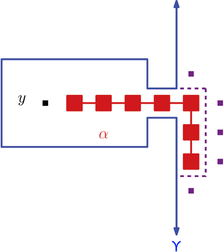}

\caption{Removal of the vertex set $\alpha$ from $V$. The enlarged
squares represent $\alpha$ and the dotted path is the segment of
$\overline{\Upsilon}$ that does not lie in $\Upsilon$. The vertices
$\bar{v}_j$ for $j \in J$ are drawn neighboring the dotted path on the right.}
\label{figalpha}
\end{figure}

If $\bar{\tau}$ and $\overline{B}_n$ are the passage times and
Busemann functions in $(\overline{V},\overline{E})$ (defined~in the
obvious way), then
%
\begin{equation}
\label{eqdudeoverlines} \overline{B}(\bar{v}_i,\bar{v}_j) =
\lim_{n \rightarrow\infty} \overline{B}_n(\bar{v}_i,
\bar{v}_j)
\end{equation}
exists for all $i$ and $j$ by Proposition \ref{propbusemannaxis}.

Denote by $J \subseteq\mathbb{Z}$ the finite set of indices such that
$\bar{v}_j$ is at Euclidean distance one from $\alpha$. Note that $y$
is adjacent to some vertex of $\alpha$; therefore, if $y \in\overline
{V}$, then $y = \bar{v}_j$ for some $j \in J$. We will want to apply
the following lemma to both $z=y$ and $z=v_m$:

\begin{lem}\label{lemnewlemma}
Let $z \in V$ be such that either $z \in\{\bar{v}_j\dvtx  j \in J\}$
or $z \notin\overline{V}$. Then
%
\begin{equation}
\label{eqfrostmanslemma} \tau(z,v_n) = \min_{j \in J} \bigl\{
\tau(z,\bar{v}_j) + \bar{\tau }(\bar{v}_j,v_n)
\bigr\}.
\end{equation}
\end{lem}

\begin{pf}
Let $\varepsilon>0$ and $j \in J$. Then find paths $\gamma\dvtx  z \leadsto
\bar{v}_j$ in $(V,E)$ and $\bar{\gamma}\dvtx  \bar{v}_j \leadsto v_n$
in $(\overline{V},\overline{E})$ such that $\tau(\gamma) \leq\tau
(z,\bar{v}_j) + \varepsilon$ and $\bar{\tau}(\bar{\gamma}) \leq
\bar{\tau}(\bar{v}_j,v_n) + \varepsilon$. Build a path $\sigma\dvtx
z \leadsto v_n$ in $(V,E)$ by concatenating $\gamma$ with $\bar
{\gamma}$. Then
\[
\tau(z,v_n) \leq\tau(\sigma) = \tau(\gamma) + \bar{\tau }(\bar{
\gamma}) \leq\tau(z,\bar{v}_j) + \bar{\tau}(\bar
{v}_j,v_n) + 2\varepsilon.
\]
Taking $\varepsilon\to0$ and a minimum over $j \in J$ gives the
inequality $\leq$ in (\ref{eqfrostmanslemma}).

To prove the other inequality, let $\sigma\dvtx  z \leadsto v_n$ in $(V,E)$
be a path such that $\tau(\sigma) \leq\tau(z,v_n) + \varepsilon$. The
path $\sigma$ must have a terminal segment $\bar{\gamma}$ which lies
in $(\overline{V},\overline{E})$ from some $\bar{v}_{j_0}$ to
$v_n$---this terminal segment may be equal to the singleton $\{v_n\}$.
Write $\gamma$ for the segment of $\sigma$ from $z$ to the last
meeting of $\bar{v}_{j_0}$. Then
\begin{eqnarray*}
\min_{j \in J} \bigl\{ \tau(z,\bar{v}_j) + \bar{\tau}(\bar {v}_j,v_n) \bigr\} &\leq&\tau(z,
\bar{v}_{j_0}) + \bar{\tau }(\bar{v}_{j_0},v_n)
\\
&\leq&\tau(\gamma) + \bar{\tau}(\bar{\gamma}) = \tau(\sigma) \leq
\tau(z,v_n) + \varepsilon.
\end{eqnarray*}
Taking $\varepsilon\to0$ proves (\ref{eqfrostmanslemma}).\vadjust{\goodbreak}
\end{pf}

So, defining
\[
\varphi_j(z,n):= \tau(z,\bar{v}_j) + \bar{\tau}(
\bar {v}_j,v_n) - \bar{\tau}(\bar{v}_1,
v_n),
\]
we see that $\tau(z,v_n) = \bar{\tau}(\bar{v}_1, v_n) + \min_{j \in J}\varphi_j(z,n)$.
Moreover,
\[
\lim_{n \rightarrow\infty} \varphi_j (z, n)=:
\varphi_j(z)
\]
exists by (\ref{eqdudeoverlines}), and therefore so does
%
\begin{equation}
\label{eqpartiallimits} \lim_{n\to\infty} \bigl[\tau(z,v_n) -
\bar{\tau}(\bar{v}_1, v_n)\bigr].
\end{equation}

Finally, we can use the above to show convergence of $B_n(y, v_m)$ as
$n \to\infty$. Write
\begin{eqnarray*}
\lim_{n \rightarrow\infty} B_n (y, v_m) &=& \lim
_{n \rightarrow
\infty} \bigl[ \tau(y,v_n) - \tau(v_m,
v_n) \bigr]
\\
&=& \lim_{n \rightarrow\infty} \bigl[ \tau(y,v_n) - \bar{\tau
}(\bar{v}_1, v_n) + \bar{\tau}(\bar{v}_1,
v_n) - \tau(v_m, v_n) \bigr]
\\
&=& \lim_{n \rightarrow\infty} \bigl[ \tau(y,v_n) - \bar{\tau
}(\bar{v}_1, v_n) \bigr] - \lim_{n \rightarrow\infty}
\bigl[ \tau (v_m, v_n) - \bar{\tau}(
\bar{v}_1, v_n) \bigr].
\end{eqnarray*}
Using (\ref{eqpartiallimits}) with $z=y$ and $z=v_m$ completes the proof.

\section{Geodesic limits}\label{sec3}
Our aim in this section is to prove Theorem~\ref{teographconvergence}. We begin with general properties of geodesic graphs
from \cite{DH12}.

\subsection{Geodesic graphs}\label{sec3.1}
We will show that the geodesic graph is in fact a union of geodesics
with the appropriate directions. Moreover, under the assumption of
unique passage times, it is a directed forest.

%
\begin{prop}\label{propGGproperties}
Assume $\mathbb{P}$ admits geodesics.
\begin{longlist}[(2)]
\item[(1)] Almost surely, every finite directed path in $\mathbb
{G}_n$ is a geodesic. It is a subpath of a geodesic ending in $v_n$.
\item[(2)] Assume $\mathbb{P}$ has unique passage times. Then each $x
\in V \setminus\{v_n\}$ has out-degree~1 in $\mathbb{G}_n$.
Furthermore viewed as an undirected graph, $\mathbb{G}_n$ has no circuits.
\end{longlist}
\end{prop}

\begin{pf}
Let $\gamma$ be a directed path in $\mathbb{G}_n$ and write the
(directed) edges of $\gamma$ in order as $e_1, \ldots, e_k$. Write $J
\subseteq\{1, \ldots, k\}$ for the set of $j$ such that the path
$\gamma_j$ induced by $e_1, \ldots, e_j$ is a subpath of a geodesic
from some vertex to $v_n$. We will show that $k \in J$. By construction
of $\mathbb{G}_n$, the edge $e_1$ is in a geodesic from some point to
$v_n$. Furthermore, if $e_1 = (x,y)$, then $\tau(x,v_n) \geq\tau
(y,v_n)$ because $\eta_n(e_1)=1$, so if these passage times are not
equal, $e_1$ must be traversed from $x$ to $y$ in this geodesic, giving
$1\in J$. If they are equal, then $\omega_{\{x,y\}}=0$ and $1 \in J$
as well.

Now suppose that $j \in J$ for some $j<k$; we will show that $j+1 \in
J$. Take $\sigma$ to be a geodesic from a point $z$ to $v_n$ which
contains $\gamma_j$ as a subpath. Write $\sigma'$ for the segment of
the path from $z$ to the far endpoint $w_j$ of $e_j$ (i.e., we
terminate $\sigma$ directly after traversing the path $\gamma_j$ for
the first time). The edge $e_{j+1}$ is also in $\mathbb{G}_n$ so it is
in a geodesic from some point to $v_n$. If we write $\hat{\sigma}$
for the piece of this geodesic from its first meeting of $w_j$ to
$v_n$, we claim that the concatenation of $\sigma'$ with $\hat{\sigma}$ is a geodesic from $z$ to $v_n$. To see this,
\[
\tau(z,v_n) = \sum_{e \in\sigma'}
\omega_e + \sum_{e \in\sigma
\setminus\sigma'}
\omega_e = \sum_{e \in\sigma'}
\omega_e + \sum_{e \in\hat{\sigma}}
\omega_e.
\]
The last equality holds since both $\hat{\sigma}$ and the segment of
$\sigma$ from $w_j$ to $v_n$ are geodesics, so they have equal passage
time. Hence $j+1 \in J$, and we are done with the first item.

For the second item, assume that $\mathbb{P}$ has unique passage times
so that in particular, almost surely, no edges have passage time 0.
Therefore if a directed edge is in a geodesic from a point to $v_n$, it
must be traversed in this direction. Note that from each vertex $v \in
V \setminus\{v_n\}$ there is at least one geodesic from $v$ to $v_n$.
The first edge of this geodesic is pointed away from $v$, so $v$ has an
out-degree of at least one. Assuming $v$ has an out-degree of at least
two, then we write $e_1$ and $e_2$ for two such directed edges. By the
first item, there are two geodesics, $\gamma_1$ and $\gamma_2$, to
$v_n$ such that $e_i \in\gamma_i$ for $i=1,2$. If either of these
paths returned to $v$, then there would exist a finite path with
passage time zero, contradicting unique passage times. So the portions
of the $\gamma_i$'s from $v$ to $v_n$ have distinct edge sets and
therefore have different passage times. This contradicts both being geodesics.

We finish by arguing for the absence of circuits. If there is a circuit
in the undirected version of $\mathbb{G}_n$, then by virtue of each
vertex having out-degree one, this is a directed circuit and thus a
geodesic. But then it has passage time zero, a contradiction.
\end{pf}

\subsection{Proof of Theorem~\texorpdfstring{\protect\ref{teographconvergence}}{1.3}}\label{sec3.2}
The second statement of the theorem follows directly from the previous
section: each directed path in $\mathbb{G}_n$ is a geodesic. So we
prove the first statement and show that for each directed edge $(x,y)$
in $\vec{E}$, with probability one the value of $\eta_n((x,y))$ is
eventually constant. Fix $x \in V$ and choose $m \in\mathbb{N}$ such
that, defining [with $d(\cdot,\cdot)$ the graph distance in $(V,E)$]
\begin{eqnarray*}
S_m &=& \bigl\{ w \in V\dvtx  d(x,w) \leq m \bigr\},
\\
\partial S_m &=& \bigl\{ w \in V\dvtx  d(x,w) = m+1 \bigr\},
\end{eqnarray*}
we have $S_m \cap\{v_i\}_i \neq\varnothing$. Setting $\alpha= S_m$,
we may apply Claim~\ref{clamcomponents} to find $(\overline
{V},\overline{E})$, a component of the graph generated by removing
$\alpha$ from $(V,E)$ containing $v_n$ for all large $n$. As before,
it can be alternatively created by cutting $(\mathbb{Z}^2,\mathcal
{E}^2)$ along a doubly infinite self-avoiding dual path $\overline
{\Upsilon}$. As in the last section, we will decorate expressions with
an overline when they are meant for the model in $(\overline{V},
\overline{E})$ (e.g.,~$\bar\tau$). For the remainder,
we also fix $\omega\in\Omega$ such that for each $x,y \in V$, there
is a geodesic from $x$ to $y$.

For each $\zeta\in T_m:= \partial S_m \cap\overline{V}$, and $n$
such that $v_n \in\overline{V}$, we define the quantity
%
\begin{equation}
\label{eqtauboxdecomp} f_n(\zeta) = \tau(x,\zeta) + \bar\tau(\zeta,
v_n).
\end{equation}
Let $\mathfrak{m}_n$ be the set of minimizers of $f_n$.

%
\begin{lem}
\label{lemuniquemin}
There exists $\mathfrak{m} \subset T_m$ such that $\mathfrak{m}_n =
\mathfrak{m}$ for all large $n$.
\end{lem}

\begin{pf}
First, note that $T_m \subset\{\bar{v}_i\}_i$. Therefore by
Proposition~\ref{propbusemannaxis}, for  \mbox{$\zeta,\zeta' \in T_m$},
\begin{eqnarray*}
f_n(\zeta) - f_n\bigl(\zeta'\bigr) &=&
\tau(x,\zeta) + \bar\tau(\zeta,v_n) - \tau\bigl(x,
\zeta'\bigr) - \bar\tau\bigl(\zeta',v_n
\bigr)
\\
&=& \tau(x,\zeta) - \tau\bigl(x,\zeta'\bigr) + \overline
B_n \bigl(\zeta,\zeta'\bigr)
\end{eqnarray*}
is eventually monotone. Suppose that $\zeta\in T_m$ satisfies $\zeta
\notin\mathfrak{m}_n$ for infinitely many~$n$. Then we can find
$\zeta'$ such that $f_n(\zeta)-f_n(\zeta') > 0$ for infinitely many
$n$. By monotonicity this means that actually $f_{n}(\zeta) -
f_{n}(\zeta') > 0$ for all large $n$ and thus $\zeta\notin\mathfrak
{m}_n$ for all large $n$. This also implies that if $\zeta\in
\mathfrak{m}_n$ for infinitely many $n$, then $\zeta\in\mathfrak
{m}_n$ for all large $n$, completing the proof.
\end{pf}

Given this lemma, the theorem will follow once we show that $\eta
_n((x,y))=1$ if and only if $\{x,y\}$ is in a geodesic from $x$ to a
vertex of $\mathfrak{m}_n$. Note that $T_m$ is equal to the set of
vertices in $\overline{V}$ at Euclidean distance one from $S_m$.
Applying Lemma~\ref{lemnewlemma} with $z=x$, any $\zeta\in T_m$ satisfies
\[
\zeta\in\mathfrak{m}_n\mbox{ if and only if } f_n(
\zeta) = \tau(x,v_n).
\]
So suppose first that $\eta_n((x,y)) = 1$; then $\{x,y\}$ is in a
geodesic $\gamma$ from $x$ to $v_n$. $\gamma$~has a last intersection
$\zeta$ with $T_m$. Then the segment $\bar{\gamma}$ of $\gamma$
from this intersection to $v_n$ has
\[
\tau(\zeta,v_n) = \tau(\bar{\gamma}) \geq\bar{\tau}(\zeta,v_n).
\]
But $\bar{\tau}(\zeta,v_n) \geq\tau(\zeta,v_n)$, so $\tau
(\bar{\gamma}) = \bar{\tau}(\zeta,v_n)$. Therefore
\[
\tau(x,v_n) = \tau(\gamma) = \tau(x,\zeta) + \tau(\bar{\gamma}) =
\tau(x,\zeta) + \bar{\tau}(\zeta,v_n) = f_n(\zeta),
\]
giving $\zeta\in\mathfrak{m}_n$. Furthermore the segment of $\gamma
$ up to the last intersection with $\zeta$ is a geodesic from $x$ to
$\zeta$ that contains $\{x,y\}$.

Conversely, suppose that $\{x,y\}$ is in a geodesic $\gamma_1$ from
$x$ to a vertex $\zeta$ of~$\mathfrak{m}_n$; we will show that $\eta
_n((x,y))=1$. Choose $\gamma_2$ as any geodesic from $\zeta$ to
$v_n$. Concatenate them to form a path $\gamma$ from $x$ to $v_n$. We compute
\[
\tau(\gamma) = \tau(\gamma_1) + \tau(\gamma_2) = \tau(x,
\zeta) + \tau(\zeta,v_n) \leq\tau(x,\zeta) + \bar{\tau}(\zeta,v_n) = f_n(\zeta).
\]
However since $\zeta\in\mathfrak{m}_n$, $f_n(\zeta) = \tau
(x,v_n)$, so $\tau(\gamma) \leq\tau(x,v_n)$. The opposite
inequality holds because $\gamma\dvtx  x \leadsto v_n$, so $\gamma$ is a
geodesic from $x$ to $v_n$. It remains to show that $\tau(x,v_n) \geq
\tau(y,v_n)$. But this holds because $y$ appears in $\gamma$ after
the first appearance of $x$. Therefore if we write $\sigma$ for the
segment of $\gamma$ from the first intersection with $y$ to $v_n$, then
\[
\tau(x,v_n) = \tau(\gamma) \geq\tau(\sigma) = \tau(y,v_n).
\]

\section{Geodesics graphs on $\mathbb{H}$}\label{sec4}
In this section we prove Theorem~\ref{teohalfplanegeodesics}.
Because $\mathbb{P}$ admits geodesics, Theorem~\ref{teographconvergence} implies that the sequence of graphs $(\mathbb
{G}_n)$ converge almost surely to a directed graph $\mathbb{G}$, each
of whose directed paths is a geodesic. As $\mathbb{P}$ also has unique
passage times, Proposition~\ref{propGGproperties} states that each
vertex of $\mathbb{G}_n$ has out-degree one and there are no
undirected circuits, so these same properties survive in the limit for
$\mathbb{G}$. The finiteness of backward clusters is a consequence of
nonexistence of bigeodesics in the half-plane, proved by Wehr and Woo
\cite{WW98}. Unfortunately this result was only proved under \textup{(A)}
with the additional assumption $\mathbb{E} \omega_e < \infty$, so we
provide a proof in Appendix~\ref{secWW} under either \textup{(A)}~or~\textup{(B)}.

This section is devoted to showing coalescence of directed paths in
$\mathbb{G}$. Because each vertex in $\mathbb{G}_H$ has an out-degree
of one, it suffices to show that each $\Gamma_v$~and~$\Gamma_w$
(defined in the statement of Theorem~\ref{teohalfplanegeodesics})
share a vertex. The main difficulty will be proving this statement for
all $v,w$ on the first coordinate axis; that is, the set~$L_0$, where
\[
\mbox{for }k \in\mathbb{N} \cup\{0\},\qquad L_k:= \bigl\{(x,k)\dvtx  x \in
\mathbb{Z}\bigr\}.
\]
To see why this implies coalescence for all paths, assume we have
proved this statement, and note that it suffices then to show that for
all $v,w \in V_H$ with $w \in L_0$, the geodesics $\Gamma_v$ and
$\Gamma_w$ coalesce. Write $v=(v_1,v_2)$ and consider the set
\[
\widetilde L_v = \bigl\{(v_1,y) \in V_H\dvtx  0
\leq y \leq v_2\bigr\}.
\]
With probability one, for each $v' \in\widetilde L_v$, the backward
cluster $B_{v'}$ is finite. Thus we can find $m,n \in\mathbb{Z}$ with
$m<v_1<n$ such that for all $v' \in\widetilde L_v$, both points $(m,0)$
and $(n,0)$ are not in $B_{v'}$. This means in particular that $\Gamma
_{(m,0)}$ and $\Gamma_{(n,0)}$ cannot intersect $\widetilde L_v$ and,
since they coalesce, they must meet ``above'' $v$. In other words,
$v$~is in the bounded component of $V_H \setminus(\Gamma_{(m,0)} \cup
\Gamma_{(n,0)})$ (viewing these paths only as their vertex sets). By
planarity, $\Gamma_v$ must intersect $\Gamma_{(m,0)}$. Because
$\Gamma_{(m,0)}$ coalesces with~$\Gamma_w$, this completes the proof.

\begin{figure}[b]

\includegraphics{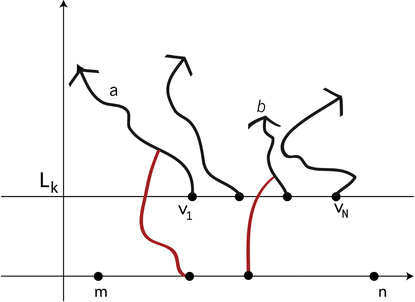}

\caption{In this example $N_{m,n}^{(k)}$ is at least $4$. The arrowed
paths are geodesics emanating from vertices on the line $L_k$. They do
not intersect each other, and they intersect $L_k$ only at their
initial points. The nonarrowed paths are segments of geodesics
starting from $L_0$. Note that the initial points of $a$ and $b$ do not
contribute to the random variable $M_{m,n}^{(k)}$.}\label{figfiguremn}
\end{figure}

So we move to proving coalescence starting from the first coordinate
axis. We will prove by contradiction, so assume either \textup{(A)}~or~\textup{(B)} but that
%
\begin{equation}
\label{eqassumption} \mbox{with positive probability, there are vertices } v,w \in
L_0\mbox{ with }\Gamma_v \cap\Gamma_w
= \varnothing.\hspace*{-25pt}
\end{equation}

\subsection{Estimates on density of disjoint geodesics}\label{sec4.1}

\subsubsection{Definitions}\label{sec4.1.1}\label{secdensitydefinitions}
For each $k \in\mathbb{N} \cup\{0\}$ and $m,n \in\mathbb{Z}$ with
$m<n$ define $N_{m,n}^{(k)}$ as the largest number $N$ such that we can
find vertices $v_1, \ldots, v_N \in[m,n] \times\{k\}$ such that:
\begin{longlist}[(a)]
\item[(a)] $\Gamma_{v_1}, \ldots, \Gamma_{v_N}$ are pairwise
disjoint, and
\item[(b)] for all $i$, $\Gamma_{v_i} \cap [L_0 \cup\cdots\cup L_k ] = \{v_i\}$.
\end{longlist}
Similarly, for $k \in\mathbb{N}$ let $M_{m,n}^{(k)}$ be the largest
$M$ such that we can find $v_1, \ldots, v_M \in[m,n] \times\{k\}$
such that (a) and (b) above hold but also (c) for all $i = 1, \ldots,M$, every $v \in L_0$ has $\Gamma_v \cap\Gamma_{v_i} = \varnothing$.
See Figure \ref{figfiguremn} for an illustration of these definitions.

\begin{lem}\label{lemkingman}
For each $k_1 \in\mathbb{N} \cup\{0\}$ and $k_2 \in\mathbb{N}$,
there exist deterministic $\alpha_{k_1},\beta_{k_2} \geq0$ such that
\[
\lim_{n \to\infty} \frac{N_{0,n}^{(k_1)}}{n} = \alpha_{k_1}\quad
\mbox{and}\quad\lim_{n \to\infty} \frac{M_{0,n}^{(k_2)}}{n} =
\beta_{k_2}\qquad\mbox{almost surely and in } L^1(\mathbb{P}).
\]
We have the characterization
\[
\alpha_{k_1} = \inf_{n \in\mathbb{N}} \frac{\mathbb{E}
N_{0,n}^{(k_1)}}{n}\quad
\mbox{and}\quad\beta_{k_2} = \inf_{n \in
\mathbb{N}}
\frac{\mathbb{E} M_{0,n}^{(k_2)}}{n}.
\]
Furthermore, assuming (\ref{eqassumption}), $\alpha_0>0$.\vadjust{\goodbreak}
\end{lem}

\begin{pf}
Note that for all $m<n<p$ in $\mathbb{Z}$ and $k_1 \in\mathbb{N}
\cup\{0\}$, $k_2 \in\mathbb{N}$, we have
\[
N_{m,p}^{(k_1)} \leq N_{m,n}^{(k_1)} +
N_{n,p}^{(k_1)}\quad\mbox {and}\quad M_{m,p}^{(k_2)}
\leq M_{m,n}^{(k_2)} + M_{n,p}^{(k_2)}.
\]
Further\vspace*{2pt} $\max\{N_{m,n}^{(k_1)},M_{m,n}^{(k_2)}\} \leq n-m+1$ surely,
so they have finite mean, and $(N_{m,n}^{(k_1)},M_{m,n}^{(k_2)})$ has
the same distribution as $(N_{0,n-m}^{(k_1)},M_{0,n-m}^{(k_2)})$.
Therefore\vspace*{1pt} we can apply Kingman's subadditive ergodic theorem to find
deterministic $\alpha_{k_1},\beta_{k_2} \geq0$ such that
\[
\frac{1}{n} N_{0,n}^{(k_1)} \to\alpha_{k_1}\quad
\mbox{and}\quad \frac{1}{n} M_{0,n}^{(k_2)} \to
\beta_{k_2}\qquad\mbox{almost surely and in } L^1(\mathbb{P}).
\]
Furthermore, $\alpha_{k_1} = \inf_{n \in\mathbb{N}} \mathbb{E}
N_{0,n}^{(k_1)}/n$ and $\beta_{k_2} = \inf_{n \in\mathbb{N}}
\mathbb{E} M_{0,n}^{(k_2)}/n$.

We claim now that under assumption (\ref{eqassumption}), $\alpha
_0>0$. By countability and invariance of $\mathbb{P}$ under
$T_{(1,0)}$, we can find $i_0 \in\mathbb{N}$ such that $\mathbb
{P}(A(1,i_0))>0$, where $A(1,i_0)$ is the event that $\Gamma_{(1,0)}$
and $\Gamma_{(i_0,0)}$ do not intersect. Note that if
$i_1<i_2<i_3<i_4$ are integers such that $\Gamma_{(i_l,0)}$ and
$\Gamma_{(i_{l+1},0)}$ are disjoint for $l=1,3$, then by planarity, at
least three of them must be disjoint. So the ergodic theorem implies
that with probability one, $A(1,i_0) \circ T_{(j,0)}$ occurs for
infinitely many $j$ and therefore we can find 4 geodesics starting from
$L_0$ that are all disjoint. The middle two of these must intersect
$L_0$ only finitely often. This implies that for some $j_0 \in\mathbb
{N}$, $\mathbb{P}(B(1,j_0))>0$, where $B(1,j_0)$ is the event that
$\Gamma_{(1,0)}$ and $\Gamma_{(j_0,0)}$ do not intersect and only
touch $L_0$ at their initial points.

Again, by the ergodic theorem,
\[
\frac{1}{N} \sum_{l=0}^N
T_{(j_0,0)}^l 1_{B(1,j_0)} \to\mathbb {P}
\bigl(B(1,j_0)\bigr)\qquad\mbox{almost surely and in }
L^1(\mathbb{P}).
\]
The reasoning given above, but applied to sets $\{j_1, j_2, \ldots\}$
of size bigger than 4, implies that for $n \in\mathbb{N}$,
\[
N_{0,j_0n}^{(0)} -1 \geq\sum_{l=0}^n
T_{(j_0,0)}^l 1_{B(1,j_0)}.
\]
Dividing by $j_0n$ and taking $n \to\infty$, we find $\alpha_0 \geq
\mathbb{P}(B(1,j_0)) / j_0 > 0$.
\end{pf}

\subsubsection{Lower bound on \texorpdfstring{$\alpha_k$}{$alpha_k$}}\label{sec4.1.2}

\begin{prop}\label{proplowerbound}
For each $k \in\mathbb{N}$, $\alpha_k \geq\beta_k + \alpha_0$.
\end{prop}

\begin{pf}
For the proof we need a lemma stating that any geodesic starting at~$L_0$ intersects $L_k$ only finitely often.

\begin{lem}\label{lemfiniteintersections}
Assume (\ref{eqassumption}). For each $v \in L_0$ and $k \in\mathbb
{N}$, with probability one, the set $\Gamma_v \cap L_k$ is finite.
\end{lem}

\begin{pf}
Assume that there exists $k \in\mathbb{N}$ such that with positive
probability, there exists $v \in L_0$ with $\Gamma_v \cap L_k$
infinite. By countability and invariance of $\mathbb{P}$ under $T_{(1,0)}$,
\[
\mathbb{P}(B) > 0\qquad\mbox{where } B = \bigl\{ \#(\Gamma_{(0,0)} \cap
L_k) = \infty \bigr\}.
\]
By Lemma~\ref{lemkingman}, we can find $N_0 \in\mathbb{N}$ such that
\[
\mathbb{P}\bigl(N_{1,N_0+1}^{(0)} > k+2\bigr) > 1-\mathbb{P}(B)/2
\]
and then by translation invariance, with positive $\mathbb
{P}$-probability, the event $B \cap\{N_{1,N_0+1}^{(0)}>k+2\} \cap\{
N_{-1-N_0,-1}^{(0)} > k+2\}$ occurs. However any outcome in this event
must have contradictory properties, as we now explain. Since $B$
occurs, $\Gamma_{(0,0)}$ must intersect infinitely many vertices of
either $L_k \cap\{(x,y)\dvtx  x \geq0\}$ or $L_k \cap\{(x,y)\dvtx  x \leq0\}
$. Let us assume the first; the subsequent argument is similar in the
other case. Then $\Gamma_{(0,0)}$ must be disjoint from at least $k+1$
different geodesics $\Gamma_{v_1}, \ldots, \Gamma_{v_{k+1}}$ with $v
_i \in L_0 \cap[1,N_0+1]$ for all $i$, but it must intersect some
vertex $(x,k)$ for $x > N_0$. By planarity, the geodesics $\Gamma
_{v_i}$ must all intersect the set $\{(x,j)\dvtx  0 \leq j \leq k\}$, but
then they cannot be disjoint. This is a contradiction.
\end{pf}

Returning to the proof of Proposition~\ref{proplowerbound}, fix $k
\in\mathbb{N}$. For each $m \in\mathbb{Z}$, define $d_k(m)$ as the
first coordinate of the last vertex (by the natural ordering) on
$\Gamma_{(m,0)}$ in the line $L_k$. This quantity exists almost surely
by Lemma~\ref{lemfiniteintersections}. For any $a,b \in\mathbb
{Z}$ with $a < b$, define the set
\[
X_{a,b} = \bigl\{j \in\mathbb{Z}\dvtx  d_k(j) \in[a,b]\bigr\}.
\]

We claim that for some fixed $N_0 \in\mathbb{N}$,
%
\begin{equation}
\label{eqclam1} \mathbb{P} \bigl( X_{-N_0,n+N_0}\mbox{ contains } [0,n]
\mbox{ for infinitely many } n \in\mathbb{N} \bigr) \geq1/2.
\end{equation}
To show this, first choose $N_0 \in\mathbb{N}$ such that $\mathbb
{P}( |d_k(0)| \leq N_0) \geq3/4$. Next note that by invariance of
$\mathbb{P}$ under $T_{(1,0)}$, $\mathbb{P}(d_k(n) \leq n + N_0) \geq
3/4$ for all $n \in\mathbb{N}$. These two events occur simultaneously
with probability at least $1/2$, so
\[
\mathbb{P} \bigl( d_k(0) \geq-N_0\mbox{ and }
d_k(n) \leq n + N_0\mbox{ for infinitely many } n
\in\mathbb{N} \bigr) \geq1/2.
\]
Last, observe that by planarity and the fact that if two $\Gamma$'s
touch, they must merge, the function $m \mapsto d_k(m)$ is monotonic.
This implies that if $d_k(0) \geq-N_0$ and $d_k(n) \leq n+N_0$, then
the set $X_{-N_0,n+N_0}$ contains $[0,n]$.

The second step is to prove that
%
\begin{equation}
\label{eqclam2} \mathbb{P} \biggl( \limsup_{n \to\infty}
\frac
{N_{0,n}^{(k)}-M_{0,n}^{(k)}}{n} \geq\alpha_0 \biggr) \geq1/4.
\end{equation}
Because $(N_{0,n}^{(k)}-M_{0,n}^{(k)})/n$ converges almost surely to
$\alpha_k-\beta_k$, this suffices to complete the proof of the
proposition. First, given $\varepsilon>0$, by Lemma~\ref{lemkingman},
pick $N_1$ such that
\[
\mathbb{P} \bigl( N_{0,n}^{(0)}/n \geq\alpha_0-
\varepsilon\mbox{ for all } n \geq N_1 \bigr) \geq3/4.
\]
On this event, for $n \geq N_1$, setting $a_n = \lfloor n(\alpha
_0-\varepsilon) \rfloor$, we may find $x_1^{(n)}, \ldots,
x^{(n)}_{a_n}$ in $[0,n]$ such that the geodesics $\Gamma
_{(x_1^{(n)},0)}, \ldots, \Gamma_{(x_{a_n}^{(n)},0)}$ are pairwise
disjoint. If, in addition, the event in (\ref{eqclam1}) occurs, then
for infinitely many $n$, all of $d_k(x_1^{(n)}), \ldots,
d_k(x_{a_n}^{(n)})$ are in $[-N_0,n+N_0]$. Note that the geodesics
emanating from each of the points $(d_k(x_i^{(n)}),k)$ are disjoint and
do not intersect $L_0 \cup\cdots\cup L_k$ except for their initial
vertices. Next, choose a maximal set $\widehat{\Gamma}_1^{(n)}, \ldots,
\widehat{\Gamma}_{M_{-N_0,n+N_0}^{(k)}}^{(n)}$\vspace*{2pt} of geodesics starting in
$[-N_0,n+N_0]\times\{k\}$ which are disjoint and intersect $L_0 \cup
\cdots\cup L_k$ only at their initial vertices, and such that no $v
\in L_0$ has $\Gamma_v \cap\widehat{\Gamma}_i^{(n)} \neq\varnothing$
for $i=1, \ldots, M_{-N_0,n+N_0}^{(k)}$. Note that these $\widehat{\Gamma
}$'s are disjoint from the geodesics starting from the points
$(d_k(x_i^{(n)}),k)$. Therefore for each $n \geq N_1$, with probability
at least $1/4$ we have $N_{-N_0,n+N_0}^{(k)} \geq a_n +
M_{-N_0,n+N_0}^{(k)}$. Thus
\[
\mathbb{P} \bigl( N_{-N_0,n+N_0}^{(k)} \geq a_n +
M_{-N_0,n+N_0}^{(k)} \mbox{ for infinitely many }n \bigr) \geq1/4.
\]
By invariance of $\mathbb{P}$ under $T_{(1,0)}$,
\[
\mathbb{P} \bigl( N_{0,n+2N_0}^{(k)} - M_{0,n+2N_0}^{(k)}
\geq a_n \mbox{ for infinitely many }n \bigr) \geq1/4.
\]
Finally, as $(n+2N_0)/n \to1$ as $n \to\infty$ and $\varepsilon$ is
arbitrary, (\ref{eqclam2}) holds.
\end{pf}


\subsubsection{Upper bound on \texorpdfstring{$\alpha_k$}{$alpha_k$}}\label{sec4.1.3}

In this section we combine the lower bound from last section with an
upper bound to conclude that $\beta_k=0$. In what follows, we will
denote by $G(x,y)$ the unique geodesic between $x$ and $y$.

%
\begin{prop}\label{propbetazero}
For $k \in\mathbb{N}$, $\alpha_k \leq\alpha_0$. Therefore by
Proposition~\ref{proplowerbound}, $\beta_k=0$.
\end{prop}

We will couple together the upper half-plane with shifted half-planes.
For any $k \in\mathbb{N}$ we consider the shifted configuration
$T_{(0,k)}\omega$ and the unique geodesics $G(v,(n,0))$ in this
configuration. Specifically, for any $\omega\in\Omega_H$ and $v \in
V_H^k = \{(x,y) \in V_H\dvtx  y \geq k\}$, we set
%
\begin{equation}
\label{eqshiftedgeodesicdef} G_n^{(k)}(v) = T_{(0,-k)} \bigl[ G
\bigl(v-(0,k),(n,0)\bigr) (T_{(0,k)} \omega ) \bigr],
\end{equation}
where for a path $\gamma$ in $\mathbb{H}$ we denote by $T_{(0,-k)}
\gamma$ the path $\gamma$ shifted up by $k$ units. By Theorem~\ref
{teographconvergence}, there is an almost sure limit $G^{(k)}(v) =
\lim_{n \to\infty} G_n^{(k)}(v)$.


%
\begin{lem}\label{lemshiftup}
Let $k \in\mathbb{N}$. With probability one, for all $v \in L_k$, if
$\Gamma_v \cap [ L_0 \cup\cdots\cup L_{k-1}  ] =
\varnothing$, then
\[
\Gamma_v = G^{(k)}(v).
\]
\end{lem}

\begin{pf}
Let $v \in L_k$ such that $\Gamma_v \cap [ L_0 \cup\cdots\cup
L_{k-1}  ] = \varnothing$ and write it as $v=(v_1,v_2)$. Let
$\sigma$ be the nonself intersecting continuous curve obtained by
concatenating (a) the edges of $\Gamma_v$, (b) the vertical line
segment connecting $(v_1,-1/2)$ and $v$ and (c) the ray $\{(x,-1/2) \in
\mathbb{R}^2\dvtx  x \geq v_1\}$. One component of the complement of
$\sigma$ contains all vertices of $L_{k-1}$ to the right of $v-(0,1)$,
and the other contains all vertices of $L_{k-1}$ to the left of
$v-(0,1)$; call the first $C_1$ and the second $C_2$. Because the
sequence $G(v,(n,0))$ converges to $\Gamma_v$ as $n \to\infty$,
there exists $N_0$ such that if $n \geq N_0$, then $G(v,(n,0))$ does
not contain any vertices of the form $(v_1,y)$ for $y<v_2$. For $n \geq
N_0$ the geodesic $G(v,(n,0))$ cannot contain any vertices in $C_2$.
For if it did, it would start at $v$, go through a vertex in $C_2$, and
then touch $(n,0)$, a vertex in $C_1$. Because this geodesic cannot
cross $\{(v_1,y)\dvtx  y < v_2\}$, it must cross $\Gamma_v$ and violate
unique passage times.

For $n \geq N_0$, let $w_n$ denote the first intersection of
$G(v,(n,0))$ with $L_{k-1}$. The vertex $v_n$ directly before this must
be in $L_k$, and the segment $\gamma_n$ of $G(v,(n,0))$ from $v$ to
$v_n$ has all vertices in $V_H^k$. Therefore writing $v_n=(a_n,k)$, we
have $\gamma_n = G_{a_n}^{(k)}(v)$. Because $\Gamma_v$ does not
intersect $L_0 \cup\cdots\cup L_{k-1}$, $\|w_n\|_1 \to\infty$.
However $w_n$ is in $C_1$, so $a_n \to+\infty$. Taking $n$ to
infinity, these segments converge to $G^{(k)}(v)$. However they
converge to $\Gamma_v$.
\end{pf}

For $n \in\mathbb{N}$, choose $r=N_{0,n}^{(k)}$ pairwise disjoint
geodesics $\Gamma_{v_1}, \ldots, \Gamma_{v_r}$ for $v_1, \ldots,
v_r \in[0,n] \times\{k\}$ such that for each $i=1, \ldots, r$,
$\Gamma_{v_i} \cap[L_0 \cup\cdots\cup L_k] = \{v_i\}$. By
Lemma~\ref{lemshiftup}, $r \leq N_{0,n}^{(0)}(T_{(0,k)}(\omega))$.
Therefore
\[
\frac{N_{0,n}^{(k)}(\omega)}{n} \leq\frac
{N_{0,n}^{(0)}(T_{(0,k)}(\omega))}{n}\qquad\mbox{for all } n \in \mathbb{N}.
\]
Taking $n \to\infty$ and using invariance of $\mathbb{P}$ under
$T_{(0,k)}$, we find $\alpha_k \leq\alpha_0$.

\subsection{Deriving a contradiction}\label{sec4.2}
In this section we will show that assuming (\ref{eqassumption}),
there exists $k \geq1$ such that $\beta_k > 0$. This will contradict
Proposition~\ref{propbetazero} and complete the proof of
coalescence starting from the first-coordinate axis.




\subsubsection{Lemmas for edge modification}\label{sec4.2.1}

The first lemma will let us apply an edge modification argument. For a
typical element $\omega$ and edge $e\in E_H$ we write $\omega=
(\omega_e,\check{\omega})$. We say an event $A \subset\Omega_H$ is
$e$-\emph{increasing} if, for all $(\omega_e,\check{\omega}) \in A$
and $r>0$, $(\omega_e+r,\check{\omega}) \in A$. The following is a
weaker version of \cite{DH12}, Lemma~6.6, and uses the upward finite
energy property.

%
\begin{lem}\label{lemmodification}
Let $\lambda>0$ be such that $\mathbb{P}(\omega_e \geq\lambda)>0$.
If $A \subset\Omega_H$ is $e$-increasing with $\mathbb{P}(A)>0$, then
\[
\mathbb{P}(A,\omega_e\geq\lambda) >0.
\]
\end{lem}

\begin{pf}
We estimate
\begin{eqnarray*}
\mathbb{P}(A,\omega_e \geq\lambda) &=& \mathbb{E} \bigl[ \mathbb
{E}\bigl[ 1_A(\omega_e,\check{\omega})
1_{\{\omega_e\geq\lambda\}} \mid\check{\omega}\bigr] \bigr]
\\
&\geq&\mathbb{E} \bigl[ 1_A(\lambda,\check{\omega}) \mathbb {P}(
\omega_e\geq\lambda\mid\check{\omega}) \bigr].
\end{eqnarray*}
Because $A$ is $e$-increasing, the variable $1_A 1_{\{\omega_e \leq
\lambda\}}$ is less than or equal to the random variable $1_A(\lambda,\check{\omega})$. Therefore if the statement of the lemma is false,
then $1_A(\lambda,\check{\omega})$ is positive on a set of positive
probability. By the upward finite energy property, $\mathbb{P}(\omega
_e \geq\lambda\mid\check{\omega})$ is positive almost surely, so
the above estimates give $\mathbb{P}(A,\omega_e \geq\lambda)>0$, a
contradiction.
\end{pf}

The second lemma is a shape theorem-type upper bound. For it, we define
%
\begin{equation}
\label{eqlambda} \lambda_0^+ = \sup\bigl\{\lambda\geq0\dvtx  \mathbb{P}(
\omega_e \geq \lambda) > 0\bigr\}.
\end{equation}

%
\begin{lem}\label{lemmodifiedshapetheorem}
Suppose that $\lambda_0^+<\infty$. There exists $c^+< \lambda_0^+$
such that
\[
\mathbb{P} \bigl( \tau(0,x) \leq c^+ \|x\|_1\mbox{ for all but
finitely many } x \in V_H \bigr) = 1.
\]
\end{lem}

\begin{pf}
Because $\mathbb{P}$ has unique passage times, the marginal of $\omega
_e$ is not concentrated at a point and therefore $\mathbb{E} \omega_e
< \lambda_0^+$. For any $x \in V_H$ choose a deterministic path
$\gamma_x\dvtx  0 \leadsto x$ in $\mathbb{H}$ with $\|x\|_1$ number of
edges. Then
\[
\mathbb{E} \tau(0,x) \leq\mathbb{E} \tau(\gamma_x) = \|x
\|_1 \mathbb{E}\omega_e.
\]
We now set $c^+ = \frac{\mathbb{E}\omega_e + \lambda_0^+}{2}$ and
argue that this value satisfies the condition of the lemma. The
argument will be similar to the proof of the shape theorem in the full space.

For any $z \in\mathbb{Q}^2$ with second coordinate nonnegative, let
$N$ be any natural number such that $Nz \in V_H$. Then for $n \in
\mathbb{N}$, write $n=\lfloor\frac{n}{N} \rfloor+ r$, where $0\leq
r < N$ and estimate
\[
\tau(0,nz) \leq N \lambda_0^+ \|z\|_1 + \sum
_{i=0}^{\lfloor
{n}/{N} \rfloor-1} \tau(0,Nz) \bigl(T^i_{Nz}
\omega\bigr).
\]
Divide by $n$ and use the ergodic theorem to find
%
\begin{equation}
\label{eqlimsup} \limsup_{n \to\infty} \frac{\tau(0,nz)}{n} \leq
\frac{\mathbb{E}
\tau(0,Nz)}{N} \leq\|z\|_1 \mathbb{E} \omega_e.
\end{equation}
Let $\Omega'_H$ be the full-probability event on which (\ref{eqlimsup}) holds for all $z \in\mathbb{Q}^2$ with second coordinate
nonnegative. Assume by way of contradiction that on some positive
probability event $A$, the lemma does not hold for the $c^+$ fixed
above. Then we can find $\omega\in A \cap\Omega'_H$; we will show
that this $\omega$ has contradictory properties.

Let $(z_n)$ be a sequence of vertices in $V_H$ such that $\|z_n\|_1 \to
\infty$ and
\[
\tau(0,z_n) > c^+ \|z_n\|_1\qquad\mbox{for
all } n \in\mathbb{N}.
\]
By compactness (and by restricting to a subsequence), given a positive
$a$ such that $a\lambda_0^+< c^+-\mathbb{E}\omega_e$, we can find
some $z \in\mathbb{Q}^2$ with second coordinate nonnegative and
\[
\|z\|_1=1\mbox{ such that } \biggl\llVert \frac{z_n}{\|z_n\|
_1}-z
\biggr\rrVert _1 < a\qquad\mbox{for all } n \in\mathbb{N}.
\]
Then we can estimate
\[
\tau(0,z_n) \leq\tau\bigl(0,\|z_n\|_1z\bigr) + \tau\bigl(
\|z_n\|_1z,z_n\bigr) \leq\tau \bigl(0,\|z_n
\|_1z\bigr) + \bigl\| \|z_n\|_1z-z_n
\bigr\| _1 \lambda_0^+.
\]
Therefore
\[
c^+ < \frac{\tau(0,z_n)}{\|z_n\|_1} \leq\frac{\tau(0,\|z_n\|
_1z)}{\|z_n\|_1} + \biggl\llVert z -
\frac{z_n}{\|z_n\|_1} \biggr\rrVert _1 \lambda_0^+.
\]
Taking limsup on the right-hand side gives $c^+ \leq\mathbb{E} \omega
_e + a \lambda_0^+$, a contradiction.
\end{pf}

The final lemma deals with spatial concentration of geodesics emanating
from the first coordinate axis. For $v_1, v_2, v_3 \in L_0$ let
$B(v_1,v_2,v_3)$ be the event that:
\begin{longlist}[(3)]
\item[(1)] the geodesics $\Gamma_{v_1}, \Gamma_{v_2}$ and $\Gamma
_{v_3}$ are disjoint;
\item[(2)] they intersect $L_0$ only at their initial points;
\item[(3)] their intersection with each $L_k$ is finite.
\end{longlist}
We will also need a subevent of $B(v_1,v_2,v_3)$. Let
\[
B^G(v_1,v_2,v_3) = \lleft
\{ \matrix{ B(v_1,v_2,v_3) \mbox{ occurs and for each } \varepsilon>0,
\vspace*{1pt}\cr
\mbox{there are infinitely many } k \in\mathbb{N} \mbox{ such that}
\vspace*{1pt}\cr
\mbox{the last intersections }\zeta_k \mbox{ and }\zeta_k'\mbox{ of}
\vspace*{1pt}\cr
\Gamma_{v_1}\mbox{ and }\Gamma_{v_3}\mbox{ with }L_k\mbox{ have } \|\zeta_k - \zeta_k'\|_1 < \varepsilon k }%
\rright\}.
\]

\begin{lem}\label{lemangularconcentration}
Suppose $v_1=(x_1,0), v_2=(x_2,0)$ and $v_3=(x_3,0)$ with
$x_1<x_2<x_3$. Then $\mathbb{P}(B^G(v_1,v_2,v_3) \mid B(v_1,v_2,v_3))
= 1$.
\end{lem}

\begin{pf}
For $z \in L_0$ and $k \in\mathbb{N}$, denote by $\zeta_k(z)$ the
last point of intersection of $\Gamma_z$ with $L_k$, which exists
almost surely by Lemma~\ref{lemfiniteintersections}. Take
$v=v_3-v_1$ and consider
\[
C_k(v) = \bigl\{\bigl\|\zeta_k(v) - \zeta_k(0)
\bigr\|_1 \geq\varepsilon k \bigr\}.
\]
For\vspace*{2pt} $k, n \in\mathbb{N}$, define $X_n^{(k)} = \sum_{j=0}^{n-1}
1_{C_k(v)}(T_{(jd,0)}(\omega))$, where $d=\|v\|_1+1$. By the ergodic
theorem, putting $p_k = \mathbb{P}(C_k(v))$,
%
\begin{equation}
\label{eqnachos22} X_n^{(k)}/n \to p_k \qquad
\mbox{almost surely}.
\end{equation}

As previously stated in the paper, for $l \in\mathbb{Z}$ and $k \in
\mathbb{N}$, define $d_k(l)$ as the first coordinate of $\zeta_k(l)$,
and note that by planarity, $d_k(l)$ is monotone in $l$. Therefore for
$n \in\mathbb{N}$, the difference $d_k(nd) - d_k(0)$ is at least
equal to $\varepsilon k X_n^{(k)}$, so
\[
\frac{d_k(nd) - nd - d_k(0)}{n} \geq\frac{\varepsilon k X_n^{(k)} -
nd}{n} = \varepsilon k X_n^{(k)}/n
- d.
\]
Combining with (\ref{eqnachos22}), almost surely,
\[
\liminf_{n \to\infty} \frac{d_k(nd)-nd-d_k(0)}{n} \geq\varepsilon k
p_k - d.
\]
Because $d_k(nd)-nd$ and $d_k(0)$ have the same distribution,
$(d_k(nd)-nd - d_k(0))/n \to0$ in probability. Therefore
\[
p_k \leq d/(\varepsilon k),
\]
giving $p_k \to0$. In particular, with probability one, $C_k(v)^c$
occurs for infinitely many $k$.
\end{pf}

\subsubsection{Main argument}\label{sec4.2.2}

We will first assume that $\lambda_0^+<\infty$ and that (\ref{eqassumption}) holds. By Proposition~\ref{proplowerbound}, $\alpha
_0>0$ and so we can find $v_1,v_2,v_3$ and $p>0$ such that $\mathbb
{P}(B(v_1,v_2,v_3))\geq p$, where this event was defined before
Lemma~\ref{lemangularconcentration}. Fix any positive
%
\begin{equation}
\label{eqepsilondef} \varepsilon< \frac{\lambda_0^+-c^+}{8\lambda_0^+}.
\end{equation}

We first define a modified event which combines conditions from the
previous section. Specifically, for $k \in\mathbb{N}$ we set $B'(k) =
B'(v_1,v_3;k)$ as the event that:
\begin{longlist}[(4)]
\item[(1)] the geodesics $\Gamma_{v_1}$ and $\Gamma_{v_3}$ are
disjoint and intersect $L_j$ in a finite set for all $j \in\mathbb{N}
\cup\{0\}$;
\item[(2)] writing $w_1=w_1(k)$ and $w_3=w_3(k)$ for the last
intersections of $\Gamma_{v_1}$ and $\Gamma_{v_3}$ with $L_k$, there
is a vertex $x^*$ in $L_k$ between $w_1$ and $w_3$ such that $\Gamma
_{x^*}$ is disjoint from $\Gamma_{v_1}$ and $\Gamma_{v_3}$, and
$\Gamma_{x^*}$ intersects $L_k$ only at $x^*$;
\item[(3)] the finite geodesics $r_1(k)$ and $r_3(k)$, defined as the
segments of $\Gamma_{v_1}, \Gamma_{v_3}$ from $L_0$ to each of $w_1$
and $w_3$ satisfy $\tau(r_i(k)) \leq c^+ \|v_i-w_i\|_1$ for $i=1,3$;
\item[(4)] $\| w_1- w_3\|_1 < \varepsilon k$. (See Figure \ref{figevent}.)
\end{longlist}

\begin{figure}

\includegraphics{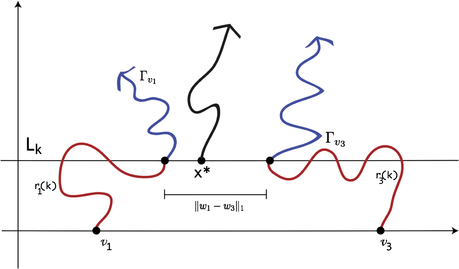}

\caption{The event $B'(k)$. The geodesics $\Gamma_{v_i}, i =1,3$, are
the left and right paths. The central geodesic $\Gamma_{x^*}$ does not
intersect either $\Gamma_{v_1}$ or $\Gamma_{v_3}$ and intersects
$L_k$ only at $x^*$. The initial segments of $\Gamma_{v_1}$ and
$\Gamma_{v_3}$ satisfy $\tau(r_i(k)) \leq c^+ \|v_i-w_i\|_1$ while
$\| w_1- w_3\|_1 < \varepsilon k$.}\label{figevent}
\end{figure}

The first two conditions hold together for all $k$ simultaneously with
probability at least $p$. This is because whenever $B(v_1,v_2,v_3)$
occurs, almost surely each $\Gamma_{v_i}$ intersects each $L_k$ in a
finite set, so we can let $x^*$ be the last intersection point of
$\Gamma_{v_2}$ with $L_k$. Next, by Lemma~\ref{lemmodifiedshapetheorem} we can find $k_0$ such that
\[
\mathbb{P}\Biggl(\tau(v_i,w) \leq c^+\|v_i-w
\|_1\mbox{ for all } i = 1, 3\mbox{ and } w \in\bigcup
_{k=k_0}^\infty L_k\Biggr) > 1-p/2.
\]
This implies that the first three conditions hold for all $k \geq k_0$
with probability at least $p/2$. Using Lemma~\ref{lemangularconcentration},
%
\begin{equation}
\label{eqbnprime} \mathbb{P}\bigl(B'(k)\bigr) > 0\qquad\mbox{for
infinitely many } k \geq k_0.
\end{equation}
We then fix any such $k \geq k_0$ with
%
\begin{equation}
\label{eqNfix} 4\|v_3-v_1\|_1
\lambda_0^+ < \frac{\lambda_0^+-c^+}{2} k.
\end{equation}

Next we modify the edge-weights for a set of edges between the
geodesics $\Gamma_{v_1}$ and $\Gamma_{v_3}$. For any configuration
$\omega$ in $B'(k)$ write $X_1$ for the closed subset of $\mathbb
{R}^2$ with boundary curves $\Gamma_{v_1}$, $\Gamma_{v_3}$ and the
segment of the first coordinate axis between $v_1$ and $v_3$. Let $X_2$
be the component of $X_1 \cap\{(x,y) \in\mathbb{R}^2\dvtx  0 \leq y \leq
k\}$ containing $v_1$. Last, define the set $X \subset E_H$ consisting
of all edges not in $\Gamma_{v_1}$ or $\Gamma_{v_3}$ but such that
both endpoints are in $X_2$. Because there are only countably many
choices, (\ref{eqbnprime}) implies there is a deterministic choice
$X'$ and a vertex $y \in L_k$ such that
%
\begin{equation}
\label{eqbnprime2} \mathbb{P}\bigl(B'(k),  X=X', x^*=y
\bigr) > 0.
\end{equation}
Here the notation $x^*=y$ means that the (deterministic) vertex $y$
satisfies condition~(2) of the definition of $B'(k)$.

We next show that
%
\begin{equation}
\label{eqpremodification} \mathbb{P} \biggl( B'(k), X=X',x^*=y,
\bigcap_{e \in X'} \biggl\{ \omega_e \geq
\frac{c^+ + \lambda_0^+}{2} \biggr\} \biggr) > 0.
\end{equation}
To prove this we enumerate the edges $e_1, \ldots, e_r$ of $X'$ and
repeatedly apply Lemma~\ref{lemmodification}. By (\ref{eqbnprime2}), we simply need to verify that for all $j=2, \ldots, r$,
\[
B'(k) \cap\bigl\{X=X',x^*=y\bigr\} \cap\bigcap
_{i=1}^{j-1} \biggl\{ \omega _{e_i}
\geq\frac{c^++\lambda_0^+}{2} \biggr\}\qquad\mbox{is $e_j$-increasing}.
\]
So take $\omega$ in the event on the left for some $j=2, \ldots, r$
with $\omega'$ such that $\omega'_f=\omega_f$ for $f \neq e_j$ and
$\omega'_{e_j} \geq\omega_{e_j}$. First we claim that $\Gamma
_{v_1}$, $\Gamma_y$ and $\Gamma_{v_3}$ are unchanged from $\omega$
to $\omega'$. To see this, note that since $e_j$ is not in $\Gamma
_{v_1}$, $\Gamma_y$ or $\Gamma_{v_3}$ we can find $n_1=n_1(\omega)$
such that if $n \geq n_1$ then $e_j$ is also not in any of the
geodesics $G(v_1,(n,0))$, $G(y,(n,0))$ or $G(v_3,(n,0))$ in $\omega$.
Therefore these remain geodesics in $\omega'$; taking the limit as $n
\to\infty$ proves the claim. Now it is clear that $X=X'$ in $\omega
'$ and conditions (1)--(4) of $B'(k)$ hold in $\omega'$. Obviously if
$\omega_{e_i}\geq(1/2) (c^+ + \lambda_0^+)$ for $i=1, \ldots, j-1$
in $\omega$, then this is still true in $\omega'$. This proves (\ref
{eqpremodification}).

On the event in (\ref{eqpremodification}), no point $v \in L_0$ can
have $\Gamma_v \cap\Gamma_y \neq\varnothing$. We will now argue for
this fact and explain why it leads to a contradiction. If such a $v$
exists, it must be on the segment of $L_0$ strictly between $v_1$ and
$v_3$; this is a direct consequence of planarity and the fact that each
vertex in $\mathbb{G}_H$ has out degree one. Therefore $\Gamma_v$
must start at $L_0$ and use only edges in $X'$ until its exit from $L_0
\cup\cdots\cup L_k$. Writing $w$ for the first vertex of $\Gamma_v$
in $L_k$, we must then have
%
\begin{equation}
\label{eqlowerboundpassagetime} \tau(v,w) \geq\frac{c^+ + \lambda_0^+}{2} \|v-w\|_1.
\end{equation}
On the other hand, we can give an upper bound for the passage time from
$v$ to $w$ by taking the path obtained by concatenating (a) the segment
of $L_0$ from $v$ to $v_1$, (b) the geodesic $r_1$ and (c) the segment
of $L_k$ from $w_1$ to $w$. We get the bound
\begin{eqnarray*}
\tau(v,w) &\leq& \bigl[ \|v_3-v_1\|_1 + \varepsilon
k \bigr] \lambda _0^+ + c^+\|v_1-w_1
\|_1
\\
&\leq&2 \bigl[ \|v_3-v_1\|_1 + \varepsilon k \bigr]
\lambda_0^+ + c^+\| v-w\|_1.
\end{eqnarray*}
Combining this with (\ref{eqlowerboundpassagetime}), we find
\[
\bigl(\lambda_0^+ - c^+\bigr)k \leq4 \bigl[ \|v_3-v_1
\|_1 + \varepsilon k \bigr] \lambda_0^+.
\]
This contradicts (\ref{eqepsilondef}) and (\ref{eqNfix}).

To summarize, we have now shown that for some fixed $w_1,w_2,w_3 \in
L_k$ such that the segment of $L_k$ between $w_1$ and $w_3$ contains
$w_2$, $C=C(w_1,w_2,w_3)$ has positive probability, where this event is
defined by the conditions:
\begin{longlist}[(2)]
\item[(1)] $\Gamma_{w_1},\Gamma_{w_2}$ and $\Gamma_{w_3}$ are
disjoint and intersect $L_0 \cup\cdots\cup L_k$ only in $w_1$, $w_2$
and $w_3$, respectively, and
\item[(2)] no $v \in L_0$ has $\Gamma_{w_2} \cap\Gamma_v \neq
\varnothing$.
\end{longlist}
Fix any $m, n \in\mathbb{Z}$ with $m<n$ and $w_1,w_3 \in[m,n] \times
\{k\}$. Let $l \in\mathbb{N}$ be bigger than $\|w_3-w_1\|_1$, and
recall the notation $M_{m,n}^{(k)}$ from Section~\ref{secdensitydefinitions}. Note that if $C \cap T_{(l,0)} C$ occurs, then
$M_{m,n + l}^{(k)} \geq2$. Iterating this reasoning, for any $j \in
\mathbb{N}$,
\[
M_{m,n+jl}^{(k)}(\omega) \geq\sum_{i=0}^{j-1}
1_C\bigl(T^i_{(l,0)} \omega \bigr).
\]
Diving by $j$ and using the ergodic theorem gives $\beta_k>0$, a
contradiction. This proves that assumption (\ref{eqassumption}) is
false in the case $\lambda_0^+<\infty$ and thus all geodesics
starting from $L_0$ coalesce.

In the case that $\lambda_0^+=\infty$, the argument is much easier,
and we will just explain the idea. If (\ref{eqassumption}) holds,
then we still find $v_1,v_2,v_3$ in $L_0$ with $v_2$ in the segment of
$L_0$ between $v_1$ and $v_3$ and such that the $\Gamma_{v_i}$'s are
disjoint and intersect $L_0$ in only $v_1,v_2$ and $v_3$. Again pick
$y$ as the last intersection point of $\Gamma_{v_2}$ with $L_1$.
Letting $S$ be the set of edges touching any vertex of $L_0$ between
$v_1$ and $v_3$ (and therefore not in $\Gamma_{v_1}$ or $\Gamma
_{v_3}$), we then modify the edge-weights for edges in $S$ to be larger
than some $C_{\mathrm{big}} > 0$. Using Lemma~\ref{lemmodification} we can
find $C_{\mathrm{big}}$ large enough so that on this event, no vertex $v$ of
$L_0$ can have $\Gamma_v \cap\Gamma_y \neq\varnothing$. As before,
this implies $\beta_1>0$, a~contradiction.

\begin{appendix}
\section{Dual edge boundary of $V$}\label{sec5}\label{secdualedge}

For any set $V_1 \subseteq\mathbb{Z}^2$, let $F$ be the edge boundary
of $V_1$,
\[
F = F(V_1) = \bigl\{\{x,y\}\dvtx  x \in V_1, y \in
V_1^c\bigr\}.
\]

%
\begin{prop}\label{propdualedge}
Let $V_1 \subseteq\mathbb{Z}^2$ be infinite, connected and such that
$V_1^c$ is infinite and connected. The dual edge set $F^*$ consists of
a single doubly infinite dual path which is nonself intersecting. That
is, it is connected and infinite, and each dual vertex $v^*$ in $W^*$,
the set of endpoints of dual edges in $F^*$, has degree exactly 2 in
the connected infinite graph $G^*= (W^*,F^*)$.
\end{prop}

\begin{pf}
Assume first that $G^*$ has a cycle. We can then extract from this
cycle a self-avoiding one, whose parametrization yields a Jordan curve.
This curve must contain a vertex of $\mathbb{Z}^2$ in its interior,
showing that either $V_1$ or $V_1^c$ must be finite, a~contradiction.

Next we prove that each dual vertex $v^* \in W^*$ has degree 2 in
$G^*$. If $v^*$ has degree 1, then it has one incident dual edge $e^*
\in F^*$, and this is dual to an edge $e \in F$. One endpoint of $e$ is
in $V_1$ and one is in $V_1^c$, but they can be connected outside of
$F$ using the 3 other edges dual to those which have $v^*$ as an
endpoint, a contradiction. This means each $v^* \in W^*$ has degree at
least 2 in $G^*$. However if $v^*$ has degree at least 3 in $G^*$, then
three such dual edges $e_1^*,e_2^*$ and $e_3^*$ incident to $v^*$ are
the first edges of disjoint self-avoiding infinite dual paths $P_1,
P_2, P_3$. These paths split $\mathbb{Z}^2$ into at least 3
components, violating the fact that $(\mathbb{Z}^2, \mathcal{E}^2)
\setminus F$ has two components.

Last we must show that $G^*$ is connected. Indeed, if $G^*$ were not
connected, it would have two components $G_1^*, G_2^*$ (and possibly
others). Since each dual vertex of $G_i^*$ must have degree two, and
since there can be no cycles, $G_1^*$ and $G_2^*$ must be disjoint,
self-avoiding, doubly infinite dual paths. But this breaks $\mathbb
{Z}^2$ into at least three components, a contradiction.
\end{pf}

\section{Existence of geodesics}\label{sec6}\label{secexistence}
In this section, we prove that if $\mathbb{P}$ is a product measure
and $x$ and $y$ are arbitrary vertices of $V$, then there almost surely
exists a (finite) geodesic between $x$ and $y$. For $V = \mathbb{Z}^2$
this was proved by Wierman and Reh \cite{RW78}; for general $d$, this
appears to be open; see the remark under Theorem~8.1.8 in \cite{Z99}.
The proof will rely on the following ``partial shape theorem.''

\begin{lem}
\label{lempartialshape}
Assume that $\mathbb{P}(\omega_e = 0) < 1/2$. Then, with probability one,
\[
\liminf_{\|x\|_1 \rightarrow\infty} \frac{\tau(0,x)}{\|x\|_1} > 0.
\]
\end{lem}

\begin{pf}
Because $(V,E)$ is a subgraph of $(\mathbb{Z}^2, \mathcal{E}^2)$, it
suffices to show the lemma in the first-passage model on $\mathbb
{Z}^2$. So let $(\omega_e)$ be a passage time realization on $\mathcal
{E}^2$, and define the truncated $\hat{\omega}_e = \min\{\omega
_e,1\}$, with $\hat{\tau}$ the passage time in the environment $(\hat
{\omega}_e)$. Then by the shape theorem (see \cite{NP96}, Theorem~1,
and the references therein), the lemma holds for $\hat{\tau}$.
However, $\tau\geq\hat{\tau}$, so we are done.
\end{pf}

\begin{teo}
Let $x$ and $y$ be elements of $V$. Then, almost surely, there exists a
geodesic $\gamma\dvtx  x \leadsto y$.
\end{teo}

\begin{pf}
The proof will be broken up into two cases, depending on the
probability that $\omega_e = 0$. In both cases, we will show that if
we write for $N \in\mathbb{N}$,
\[
\tau_N(x,y) = \mathop{\min_{\gamma\dvtx  x \leadsto y}}_{\gamma
\subseteq (x + [-N,N]^2  )\cap V}
\tau(\gamma),
\]
then
%
\begin{equation}
\label{eqfinitesetmin} \mathbb{P}\bigl(\tau_N(x,y) = \tau(x,y)\mbox{ for all
large } N\bigr) = 1.
\end{equation}
This suffices to prove the theorem, as a function on a finite set
attains its minimum.
\begin{longlist}[\textit{Case} I:]
\item[\textit{Case} I:] $\mathbb{P}(\omega_e = 0) < 1/2$. In this case, we fix
some deterministic path $\gamma_0$ in $V$ connecting $x$ and $y$ and
define $N = N(\tau(\gamma_0))$ to be the smallest number such that
\[
\min_{z \in V \setminus(x + [-N,N]^2)} \tau(x,z) > \tau(\gamma_0).
\]
Note that $N$ is almost surely finite by Lemma~\ref{lempartialshape}. Then no path containing a vertex of $V \setminus
(x+ [-N,N]^2)$ can have passage time less than or equal to $\tau
(x,y)$. In particular, (\ref{eqfinitesetmin}) holds.

\item[\textit{Case} II:] $\mathbb{P}(\omega_e = 0) \geq1/2$. Choose a
deterministic $N_0>1$ such that there exists a path connecting $x$ and
$y$ lying\vspace*{1pt} entirely in $[-N_0, N_0]^2 \cap V$. We will consider $\mathbb
{P}$ to actually be defined on $\mathbb{R}^{\mathcal{E}^2}$, though
of course the weights of edges outside of $E$ will have no bearing on
the first-passage model in $(V,E)$.

Consider a sequence of annuli $A_n \subseteq\mathbb{R}^2$ of the form
\[
A_n = \bigl[-N_0^{n+1},N_0^{n+1}
\bigr]^2 \setminus\bigl(-N_0^{n},
N_0^n\bigr)^2;
\]
denote by $G_n$ the event that there is a (vertex) self-avoiding
circuit $\alpha$ in $A_n$ of edges~$e$ such that $\omega_e = 0$. By
the RSW theorem for independent percolation (see~\cite{BR06}, Section~3.1), we have
\[
\mathbb{P} \Biggl( \bigcup_{n=1}^\infty
G_n \Biggr)=1.
\]

For any $N \in\mathbb{N}$ write $L_N = N_0^{N+1}$. For a given
$\omega$ such that $G_N$ occurs, choose $\alpha$ as above, and
consider it as a continuous plane curve. Further, let $\gamma$ be any
vertex self-avoiding\vspace*{1pt} path in $(V,E)$ from $x$ to $y$. We will show that
there exists another path $\gamma'$ in $[-L_N,L_N]^2$ from $x$ to $y$
such that $\tau(\gamma') \leq\tau(\gamma)$. This suffices to
complete the proof. To do so, we use the following construction. Let
$\beta$ be any path from $x$ to $y$ in $(V,E)$ lying entirely in
$[-N_0,N_0]^2$. Since $\gamma$ intersects $\beta$ at $x$ and $y$ we
may list their common vertices in order (along $\gamma$) as $x=x_1,
\ldots, x_k=y$. We proceed along $\gamma$ from each $x_i$ to
$x_{i+1}$, calling this subpath $\gamma_i$. If $\gamma_i$ is not just
one edge of $\beta$, we create a Jordan curve $C$ by concatenating the
portion of $\beta$ from $x_i$ to $x_{i+1}$ with $\gamma_i$. If
$\alpha$ intersects the interior of $C$, then we choose any common
point $p$ and proceed in both directions along $\alpha$ from it. In
each direction we must meet $C$ again; otherwise $\alpha$ was in the
interior of $C$, which is false. Furthermore we meet $C$ before we meet
$\Upsilon$, since $\Upsilon$ is in the exterior of $C$. Therefore the
component of $\alpha\cap\operatorname{int} C$ containing $p$ is a
segment of $\alpha$ from some vertex $a$ to another $b$. Since $a$ and
$b$ are in $C$, they must be in $\gamma_i$, and we can replace the
segment of $\gamma_i$ from $a$ to $b$ with this segment of $\alpha$.
In this way we obtain a new path we call $\tilde{\gamma}_i$ and
corresponding Jordan curve $\widetilde{C}$. Note that $\tau(\tilde{\gamma}_i) \leq\tau(\gamma_i)$. See Figure~\ref
{figgeodesicsexist} for a depiction of this procedure.

\begin{figure}

\includegraphics{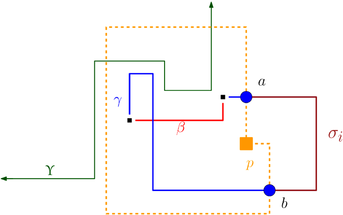}

\caption{Modifying the path $\gamma$ by replacing a segment $\sigma
_i$ of $\gamma$ with a segment of~$\alpha$. In the figure, $\alpha$~is the dotted path and $p$ is a point on $\alpha$ in the interior of
$C$, the Jordan curve formed by the union of $\gamma_i$ with $\beta$.}
\label{figgeodesicsexist}
\end{figure}

It remains to show that the procedure defined above eventually
terminates in some path $\hat{\gamma}_i$ and Jordan curve $\widehat{C}$.
At this point $\alpha$ will not intersect the interior of $\widehat{C}$,
implying that $\hat{\gamma}_i$ does not leave $[-L_N,L_N]^2$. To
prove this, assume that $p \in\alpha\cap\operatorname{int}C$ and
define $a$ and $b$ as above. Let $\sigma_i$ be the segment of $\gamma
_i$ from $a$ to~$b$. If $\sigma_i$ does not leave $\alpha$, then it
must be the complementary segment of $\alpha$ from $a$ to~$b$,
implying that $\alpha\subset(C \cup\operatorname{int}C)$. Then
$\operatorname{int}\alpha\subset\operatorname{int} C$, a
contradiction, since $\beta$ is in the interior of $\alpha$.
Therefore we can find some edge adjacent to $\alpha$ in $\sigma_i$.
When we construct $\hat{\gamma}_i$, we remove this edge from $\gamma
_i$ and only add edges of $\alpha$. Since there are only finitely many
edges adjacent to $\alpha$, the process terminates.\quad\qed
\end{longlist}\noqed
\end{pf}

\section{Absence of bigeodesics in $\mathbb{H}$}\label{sec7}\label{secWW}
In this section we outline the modifications needed to carry over the
proof of the main theorem of \cite{WW98} to our setting. An infinite
geodesic indexed by $\mathbb{Z}$ is called a bigeodesic. When we
assume unique passage times, such a path is (vertex) self-avoiding.

\subsection{Lemmas from Wehr--Woo}\label{sec7.1}
Assume either \textup{(A)} or \textup{(B)}, and let $K^*$ be the event
\[
K^* = \{\mbox{there exists a bigeodesic}\}.
\]
Note that for all $x$, $\mathbb{P}(\#B_x = \infty, (K^*)^c) = 0$,
where $B_x$ was defined in Theorem~\ref{teohalfplanegeodesics}. By
horizontal translation ergodicity, $\mathbb{P}(K^*)$ is zero or one;
let us assume for a contradiction that $\mathbb{P}(K^*)=1$.

Any bigeodesic $\gamma$ divides $\mathbb{R}^2 \setminus\gamma$ into
two components, say $R^+=R^+(\gamma)$ and $R^- = R^-(\gamma)$; that is,
\begin{eqnarray*}
R^+(\gamma) \cap R^-(\gamma) &=& \varnothing,
\\
R^+(\gamma) \cup R^-(\gamma) &=& \mathbb{R}^2 \setminus\gamma,
\\
\partial R^+ &=& \partial R^- = \gamma,
\end{eqnarray*}
where $R^-$ is a region that contains $(0,-1)$ and where $\partial A$
denotes the usual boundary of a set $A \subset\mathbb{R}^2$. Hence by
unique passage times, for any points $x,y \in R^-(\gamma)$, no bond
$b$ belonging to the finite geodesic $G(x,y)$ can be an element of
$R^+(\gamma)$. The following is \cite{WW98}, Proposition~4.

\begin{prop}\label{proplowestdef}
Consider the sequence $G((-n,0),(n,0))$ for $n \in\mathbb{N}$. With
probability 1, this sequence has a limit
\[
\gamma_0 = \lim_{n \to\infty} G\bigl((-n,0),(n,0)\bigr).
\]
Moreover, $\gamma_0$ is a bigeodesic, and for any bigeodesic $\gamma$,
\[
\gamma_0 \subset \bigl[ R^-(\gamma) \cup\gamma \bigr].
\]
\end{prop}

\begin{pf}
The same proof as in \cite{WW98} works here. The only assumption
needed is that of unique passage times.
\end{pf}

The next is \cite{WW98}, Lemma~5.

%
\begin{lem}\label{lemleavestrip}
Let $n \in\mathbb{N}$ and $\mathbb{H}' = \{(x_1,x_2) \in\mathbb
{R}^2\dvtx  x_2 \leq n\}$. With probability 1, for any bigeodesic $\gamma$
intersecting $z=(z_1,z_2)$ with $z_2 < n$,
\[
\mathbb{H}' \cap R^+(\gamma) \neq\varnothing\mbox{ and all its
components are bounded}.
\]
The boundary of each component is a self-avoiding loop, which is a
bond-disjoint union of segments of $\gamma$ and segments of the
boundary of $\mathbb{H}'$.
\end{lem}

\begin{pf}
Because we do not assume independence of the variables $(\omega_e)$,
we must modify the proof of \cite{WW98}, replacing independence with
the upward finite energy property.

In order to prove the boundedness of each component of $\mathbb{H}'
\cap R^+(\gamma)$, it is sufficient to prove that
%
\begin{equation}
\label{eqnostrip} \mathbb{P} \bigl( \mbox{there is a bigeodesic with an infinite
connected part in } \mathbb{H}'\bigr) = 0.\hspace*{-20pt}
\end{equation}

For each $k \in\mathbb{Z}$ consider a rectangular box
\[
C_k = C_k(m,n) = \bigl\{(x_1,x_2)\dvtx  2km \leq x_1 \leq(2k+1)m, 0 \leq x_2 \leq n\bigr\}.
\]
Let $T_k$ be the minimum passage time of all paths in $C_k$ which start
at a vertex in the left boundary of $C_k$ and end at a vertex in the
right boundary of $C_k$, without intersecting the top boundary. Let
$\widehat{C}_k$ for the set of edges in $\partial C_k$ that do not lie on
the first coordinate axis; then set
\[
E_k = \biggl\{ \sum_{e \in\widehat{C}_k}
\tau_e < T_k \biggr\}.
\]

We claim that for some $m$ large enough, $\mathbb{P}(E_k) > 0$ for all
$k$. To prove this, we consider two cases. Assume first that $\lambda
_0^+$, defined in (\ref{eqlambda}), is finite. Then by the ergodic
theorem, writing $e_k = \{(k,0),(k+1,0)\}$, $(1/m) \sum_{k=0}^{m-1}
\omega_{e_k} \to\mathbb{E}\omega_e$. Therefore, using the bound
$\omega_e \leq\lambda_0^+$,
\[
\lim_{m \to\infty} \frac{1}{m} \sum
_{e \in\widehat{C}_0} \omega_e = \mathbb{E}
\omega_e.
\]
As $\mathbb{P}$ has unique passage times, $\mathbb{E} \omega_e <
\lambda_0^+$, so choose $m$ such that
\[
\mathbb{P} \biggl( \sum_{e \in\widehat{C}_0} \omega_e
< \frac{\mathbb
{E} \omega_e + \lambda_0^+}{2} m \biggr) > 0.
\]
Writing $C_k^0$ for the set of edges with an endpoint in $C_k \setminus
\widehat{C}_k$, we see that the above event is $e$-increasing for all $e
\in C_0^0$. So by Lemma~\ref{lemmodification},
\[
\mathbb{P} \biggl( \sum_{e \in\widehat{C}_0} \omega_e
< \frac{\mathbb
{E} \omega_e + \lambda_0^+}{2} m, \omega_f \geq\frac{\mathbb{E}
\omega_e + \lambda_0^+}{2}
\mbox{ for all } f \in C_0^0 \biggr) > 0.
\]
On this event, each path which passes from the left to the right-hand side
of $C_0$, taking only edges in $C_0^0$, must have passage time at least
$\frac{\mathbb{E} \omega_e + \lambda_0^+}{2} m$. So for such $m$,
horizontal translation invariance gives $\mathbb{P}(E_k)>0$.

In the case that $\lambda_0^+=\infty$, the proof of $\mathbb
{P}(E_k)>0$ is easier. We simply modify the edge-weights for edges in
$C_0^0$ to be larger than the sum of the boundary edge-weights with
positive probability. In either case, the ergodic theorem shows that
\[
\mathbb{P}(E_k \mbox{ occurs for infinitely many } k >0\mbox{ and }
k < 0) = 1.
\]
For any $k$ such that $E_k$ occurs, no geodesic can pass from the
left-hand to the right-hand side of $C_k$ taking only edges in $C_k^0$,
because we can replace the segment between the left-hand and right-hand
sides by a portion of the boundary $\partial C_0$. This shows (\ref
{eqnostrip}). The rest of the lemma follows immediately.
\end{pf}

We now move to \cite{WW98}, Proposition~6, the main observation
showing that unique passage times implies that $\gamma_0$ must
intersect any large box with probability bounded below uniformly of the
position of the box. For $l \in\mathbb{N}$, let us write $B=B(l) =
[-l,l] \times[0,2l]$, and let $K$ be the event that at least one
bigeodesic intersects $B$. Define for $L \in\mathbb{N}$, translations
of $B$ by
\[
B_{i,j} = B_{i,j}(l,L) = B + (iL,jL)\qquad\mbox{for } (i,j)
\in V_H.
\]
For $L > 2l$, the $B_{i,j}$ are mutually disjoint.

%
\begin{prop}\label{prophitallboxes}
Let $\delta= 1-\mathbb{P}(K)$. Then
\begin{eqnarray*}
\mathbb{P}\bigl(B_{i,j} \subset R^+(\gamma_0)\bigr) &\leq&
\delta,
\\
\mathbb{P}\bigl(B_{i,j} \subset R^-(\gamma_0)\bigr) &\leq&
\delta.
\end{eqnarray*}
\end{prop}

\begin{pf}
The proof is the same as that in \cite{WW98}.
\end{pf}

\subsection{Main modifications}\label{sec7.2}

From this point on we must obtain a contradiction in a different manner
than what was used in \cite{WW98}; this is because the large deviation
estimate \cite{WW98}, Lemma~9, does not necessarily hold in our setting.

A consequence of Proposition~\ref{prophitallboxes} is that for any
$i,j$, $\mathbb{P}(B_{i,j} \cap\gamma_0 \neq\varnothing) > 1-2\delta
$. So using $\mathbb{P}(K^*) = 1$, choose $l$ large enough that
$1-2\delta>0$ and fix $L = 2l + 1$. For any $n \in\mathbb{N}$ let
$N_n$ be the number of boxes $B_{i,j}$ contained in $R_n:= [-l,nL+l]
\times[0,nL+2l]$ such that $B_{i,j} \cap\gamma_0 \neq\varnothing$
(the maximum number is $n^2$). The choice of~$l$ ensures that there is
a constant $c_1$ with $0<c_1\leq1$ such that
\[
\mathbb{E} N_n \geq c_1n^2\qquad\mbox{for
all }n \in\mathbb{N}.
\]
Therefore writing $\mathcal{E}_n$ for the set of edges with both
endpoints in $R_n$, for some $c_2>0$,
%
\begin{equation}
\label{eqlongpaths} \mathbb{E} \# \gamma_0 \cap\mathcal{E}_n
\geq c_2n^2\qquad\mbox {for all } n \in\mathbb{N}.
\end{equation}
%
We can then argue the following.

%
\begin{lem}\label{lemlongpaths}
Assuming $\mathbb{P}(K^*)=1$, there exists $c_3>0$ such that with
positive probability, for an infinite number of $n \in\mathbb{N}$,
there are vertices $v_1,v_2 \in\partial R_n$ such that the geodesic
$G(v_1,v_2)$ contains at least $c_3n^2$ edges in $\mathcal{E}_n$ with
weight at least~$c_3$.
\end{lem}

\begin{pf}
Let $a>0$ and choose $C>a$ such that $\# \mathcal{E}_n \leq Cn^2$ for
all $n \in\mathbb{N}$. Use~(\ref{eqlongpaths}) to estimate
\[
c_2n^2 \leq an^2 + \bigl(Cn^2-an^2
\bigr) \mathbb{P}\bigl(\# \gamma_0 \cap\mathcal {E}_n
\geq an^2\bigr),
\]
giving
%
\begin{equation}
\label{eqend1} \mathbb{P} \bigl(\# \gamma_0 \cap
\mathcal{E}_n \geq an^2 \bigr) \geq\frac{c_2-a}{C-a}.
\end{equation}
Furthermore for $b>0$, writing $p_b = \mathbb{P}(\omega_e < b)$, and
$N_n' = \# \{e \in\mathcal{E}_n\dvtx  \omega_e < b\}$,
\[
\# \mathcal{E}_n p_b = \mathbb{E} N_n'
\geq\sqrt{p_b}n^2 \mathbb {P}\bigl(N_n'
\geq\sqrt{p_b}n^2\bigr),
\]
so $\mathbb{P}(N_n' \geq\sqrt{p_b}n^2) \leq\frac{\#\mathcal{E}_n
\sqrt{p_b}}{n^2}$.\vadjust{\goodbreak}

Because $\mathbb{P}$ has unique passage times, $\mathbb{P}(\omega_e
= 0) = 0$ and so $p_b \to0$ as $b \to0$. Thus $\mathbb{P}(N_n' \geq
\sqrt{p_b}n^2) \to0$ uniformly in $n$ as $b \to0$. Combining this
with (\ref{eqend1}), choosing $a$ and $b$ small enough,
\[
\mathbb{P}\bigl(N_n' < a n^2/2
\mbox{ and }\# \gamma_0 \cap \mathcal{E}_n \geq a
n^2\bigr) > \frac{c_2}{2C}\qquad\mbox{for all } n \in\mathbb{N}.
\]
With probability at least $c_2/(2C)$, this event occurs for infinitely
many $n$ and gives at least $an^2/2$ edges in $\gamma_0 \cap\mathcal
{E}_n$ with weight at least $b$, so set $c_3 < \min\{c_2/(2C),a/2\}$.
For such an $n$, we take $v_1$ and $v_2$ to be the first and last
vertices that $\gamma_0$ touches in $R_n$.
\end{pf}

To contradict Lemma~\ref{lemlongpaths}, we will need to handle
assumptions \textup{(A)} and \textup{(B)} differently.

\subsubsection{Contradiction under \textup{(A)}}\label{sec7.2.1}

Because assumption \textup{(A)} does not include a moment condition on the
variable $\omega_e$, we will need to define modified passage times
similarly to \cite{CD81}. Choose any $D>0$ such that
\[
\mathbb{P}(\omega_e > D) \leq1/5
\]
and define a percolation process by setting $\eta_D = \eta_D(\omega)
\in\{0,1\}^{V_H}$ to be
\[
\eta_D(e) = \cases{ 0, &\quad if $\omega_e > D$,
\cr
1,
&\quad if $\omega_e \leq D$.}
\]
Because the weights $(\omega_e)$ are i.i.d., so are the variables
$(\eta_D(e))$. Because the critical value for bond percolation on
$\mathbb{Z}^2$ is $1/2$, this is a supercritical percolation process.
The following lemma holds for any $D$ such that $(\eta_D(e))$ is
supercritical, but we will give a simple proof for $D$ as above. For
the statement, we define an open half-circuit to be a path in $\mathbb
{H}$ whose initial and final endpoints are on the first coordinate axis
and all of whose edges $e$ have $\eta_D(e)=1$.

\begin{lem}
Define $B_n$ as the box $[-n,n] \times[0,2n]$ and $A_n$ as the
half-annulus $A_n = B_n \setminus B_{n-\sqrt n}$. Then
\[
\sum_n \mathbb{P}\bigl(\mbox{there is no open
half-circuit of edges in } A_n \mbox{ enclosing }(0,0)\bigr) < \infty.
\]
\end{lem}

\begin{pf}
We will consider the dual half-plane lattice $\mathbb{H}^*$, whose
vertex set is $V_H^* = V_H -(1/2,1/2)$ and whose edge set is $E_H^* =
[ E_H \setminus X ] - (1/2,1/2)$, where $X$ is the set of
edges joining vertices on the first coordinate axis. The configuration
$\eta_D$ induces one on the dual lattice $\eta_D^*$, where we set
$\eta_D^*(e^*) =1$ if $\eta_D(e)=1$ and $0$ otherwise. Here $e^*$ is
the edge dual to $e \in E_H$; that is, the unique dual edge which
bisects $e$. Note that $\eta_D^*$ has a product distribution with
\mbox{$\mathbb{P}(\eta_D^*(e^*)=1) = \mathbb{P}(\omega_e \leq D)$}.

For $v \in V_H^*$ and $n \in\mathbb{N}$, let $F_n(v)$ be the event
that there is a dual path of $n$ dual edges $e^*$ starting at $v$
satisfying $\eta_D^*(e^*)=0$ for all $e^*$. Then
\[
\mathbb{P} \bigl(F_n(v)\bigr) \leq\sum_{|P|=n}
\mathbb{P}(\omega_e > D)^n \leq\bigl(4\mathbb{P}(
\omega_e>D)\bigr)^n \leq(4/5)^n,
\]
where the sum is over all dual paths $P$ starting at $v$ with length
$n$. Therefore, letting $\partial_n^*$ be the set of dual vertices in
$B_n$ within Euclidean distance 1 of $\partial B_n$,
\[
\sum_n \sum_{v \in\partial_n^*}
\mathbb{P}\bigl(F_{\sqrt n}(v)\bigr) < \infty.
\]
But if there is no open half-circuit of edges in $A_n$ enclosing
$(0,0)$, then there is a dual path with all dual edges $e^*$ satisfying
$\eta_D^*(e^*) = 0$ starting at a dual vertex in~$\partial_n^*$ and
ending in $B_{n-\sqrt n}$.
\end{pf}

\begin{figure}[b]

\includegraphics{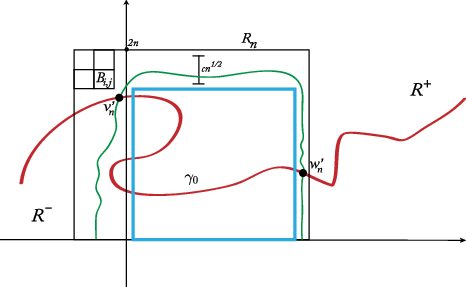}

\caption{The\vspace*{2pt} annulus $R_n \setminus [ B_{nL/2+l - \sqrt
{nL/2+l}} + (nL/2+l,0) ]$. The open half-circuit $\mathcal{C}_n$
is between the two half-boxes, and the geodesic $\gamma_0$ is the bold
path entering and leaving the large box. The first intersection of
$\gamma_0$ with $\mathcal{C}_n$ is $v_n'$ and the last intersection
is $w_n'$. Because $\gamma_0$ intersects order $n^2$ number of edges
in the inner half-box with weight at least $c_3$, $\tau(v_n',w_n')$ is
at least order $n^2$.}
\label{figevent6}
\end{figure}

Note that there is some $C>0$ such that, if $v,w$ are vertices in such
an open half-circuit mentioned in the previous lemma, then
%
\begin{equation}
\label{eqpassageupperbound} \tau(v,w) \leq CDn^{3/2}.
\end{equation}
Combining this with Lemma~\ref{lemlongpaths}, we see that with
positive probability, for infinitely many $n$, both of the following occur:
\begin{longlist}[(2)]
\item[(1)] there exist $v_n, w_n \in\partial R_n$ such that the
geodesic $G(v_n,w_n)$ contains at least $c_3n^2$ edges in $\mathcal
{E}_n$ with edge-weight at least $c_3$, and
\item[(2)] the annulus $R_n \setminus [ B_{nL/2+l - \sqrt
{nL/2+l}} + (nL/2+l,0) ]$ contains an open half-circuit $\mathcal
{C}_n$ of edges enclosing $(nL/2+l,0)$. (See Figure \ref{figevent6}.)\vadjust{\goodbreak}
\end{longlist}
Note that the above annulus contains only order $n^{3/2}$ edges total.
Therefore when these two conditions hold for large $n$, the geodesic
$G(v_n,w_n)$ must contain at least $c_3n^2/2$ edges in $B_{nL/2+l-\sqrt
{nL/2+l}} + (nL/2+l,0)$ with weight at least $c_3$. This means that
this geodesic must intersect $\mathcal{C}_n$ and contain at least
$c_3n^2/2$ edges with weight at least $c_3$ between two intersections
with $\mathcal{C}_n$. Consequently, there exist vertices $v_n'$ and
$w_n'$ on $\mathcal{C}_n$ such that $\tau(v_n',w_n') \geq
c_3^2n^2/2$. This contradicts (\ref{eqpassageupperbound}) for
large $n$.

\subsubsection{Contradiction under \textup{(B)}}\label{sec7.2.2}
Lemma~\ref{lemlongpaths} implies that with positive probability,
for infinitely many $n$, there are two vertices $v,w$ in $\partial R_n$
such that $\tau(v,w) \geq c_3^2n^2$. But this passage time is bounded
above by the sum of edge\ weights for edges in $\partial R_n$, and we find
\begin{eqnarray*}
\mathbb{P}\bigl(\tau(v,w) \geq c_3^2n^2
\mbox{ for some } v,w \in \partial R_n\bigr) &\leq&\mathbb{P}
\biggl(\sum_{e \in\partial R_n} \omega_e \geq
c_3^2n^2\biggr)
\\
& \leq& \frac{1}{c_3^4n^4} \mathbb{E} \biggl( \sum_{e \in\partial
R_n}
\omega_e \biggr)^2 = O\bigl(n^{-2}\bigr)
\end{eqnarray*}
as $n \to\infty$. Borel--Cantelli then contradicts Lemma~\ref{lemlongpaths}.
\end{appendix}

\section*{Acknowledgments}
The authors thank the Courant Institute for
hospitality. Antonio Auffinger thanks Princeton University for accommodation and
support during visits. Michael Damron thanks C. Newman for summer funding and
Jack  Hanson M. Aizenman for funding and support.


%

\printaddresses

\end{document}